\documentclass{amsart}
\usepackage{graphicx}
\usepackage{amstext,amssymb}
\newcommand{\norme}[1]{\left\Vert #1\right\Vert}
\newcommand{\Ga}{\Gamma}
\newcommand{\pa}{\partial}
\newcommand{\al}{\boldsymbol{\alpha}}
\newcommand{\sig}{\boldsymbol{\sigma}}
\newcommand{\Sig}{\boldsymbol{\Sigma}}
\newcommand{\ta}{\boldsymbol{\tau}}
\newcommand{\ch}{\boldsymbol{\chi}}
\newcommand{\ph}{\boldsymbol{\phi}}
\newcommand{\ps}{\boldsymbol{\psi}}
\renewcommand{\i}{{\rm\mathbf i}}
\DeclareMathOperator{\re}{{Re}} \DeclareMathOperator{\im}{{Im}}
\newcommand{\bR}{\mathbf{R}}
\newcommand{\Ome}{\Omega}
\newcommand{\nab}{\nabla}
\newcommand{\Del}{\Delta}
\newcommand{\cM}{\mathcal{M}}
\newcommand{\cT}{\mathcal{T}}
\newcommand{\cE}{\mathcal{E}}
\newcommand{\Div}{{\rm div}}
\newcommand{\Langle}{\left\langle}
\newcommand{\Rangle}{\right\rangle}
\def\vec#1{\mathbf{#1}}

\numberwithin{equation}{section}
\newtheorem{theorem}{Theorem}[section]
\newtheorem{proposition}{Proposition }[section]
\newtheorem{definition}{Definition }[section]
\newtheorem{lemma}{Lemma}[section]

\newtheorem{remark}{Remark}[section]

\begin{document}

\title[LDG methods for the Helmholtz equation]{Absolutely stable
local discontinuous Galerkin methods for the Helmholtz
equation with large wave number
}

\author{Xiaobing Feng}
\address{Department of Mathematics \\
         The University of Tennessee \\
         Knoxville, TN 37996, U.S.A.}
\email{xfeng@math.utk.edu}

\author{Yulong Xing}
\address{Department of Mathematics \\
The University of Tennessee \\ Knoxville, TN 37996 \\
Computer Science and Mathematics Division\\
Oak Ridge National Laboratory, Oak Ridge, TN 37830, U.S.A.}
\email{xingy@math.utk.edu}

\thanks{The work of the first author was partially supported by the NSF
grants DMS-0710831 and DMS-1016173. The research of the second author was partially
sponsored by the Office of Advanced Scientific Computing Research;
U.S. Department of Energy. The work of the second author was performed at the ORNL,
which is managed by UT-Battelle, LLC under Contract No.
DE-AC05-00OR22725.}



\begin{abstract}
Two local discontinuous Galerkin (LDG) methods
using some non-standard numerical fluxes are developed
for the Helmholtz equation with the first order absorbing boundary
condition in the high frequency regime. It is shown
that the proposed LDG methods are absolutely stable (hence well-posed)
with respect to both the wave number and the mesh size.
Optimal order (with respect to the mesh size) error estimates
are proved for all wave numbers in the preasymptotic regime.
To analyze the proposed LDG methods, they are recasted and treated as
(non-conforming) mixed finite element methods. The crux of
the analysis is to establish a generalized {\em inf-sup}
condition, which holds without any mesh constraint,
for each LDG method. The generalized {\em inf-sup}
conditions then easily infer the desired absolute stability
of the proposed LDG methods. In return, the stability results
not only guarantee the well-posedness of the LDG methods but also
play a crucial role in the derivation of the error estimates.
Numerical experiments, which confirm the theoretical results
and compare the proposed two LDG methods, are also presented
in the paper.
\end{abstract}

\maketitle

\section{Introduction}\label{sec-1}

This paper is the third installment in a series \cite{fw08a,fw08b}
which devote to developing absolutely stable discontinuous Galerkin
(DG) methods for the following prototypical Helmholtz problem
with large wave number:
\begin{alignat}{2}\label{helm-1}
-\Delta u- k^2 u &=f &&\qquad \text{in }\Omega \subset \mathbf{R}^d,\,d=2,3,  \\
\frac{\pa u}{\pa \vec{n}_{\Ome}}+\i ku &=g &&\qquad \text{on }
\Ga=\pa\Ome,  \label{helm-2}
\end{alignat}
where $\i =\sqrt{-1}$ denotes the imaginary unit. $k\in \mathbf{R}_+$ is a
given positive (large) number and known as the wave number.
\eqref{helm-2} is the so-called first order absorbing boundary
condition \cite{em79}.

We recall that \cite{fw08a,fw08b} focused on designing and analyzing
$h$- and {\em $hp$-interior penalty discontinuous Galerkin} (IPDG) methods
which are absolutely stable (with respect to wave number $k$ and mesh
size $h$) and optimally convergent (with respect to $h$). The main ideas
of \cite{fw08a,fw08b} are to introduce some novel interior penalty
terms in the sesquilinear forms of the proposed IPDG methods and
to use a non-standard analytical tool, which is based on a Rellich
identity technique, to prove the desired stability and error estimates.
The numerical experiment results shown that the absolutely stable
IPDG methods significantly outperform the standard finite element
and finite difference methods, which are known only to be stable
under stringent mesh constraints $hk\lesssim 1$ or $hk^2\lesssim 1$
(cf. \cite{dssb93,ib95a}), for the Helmholtz problem.
Moreover, the numerical experiment results also shown that these IPDG
methods are capable to correctly track the phases of the highly oscillatory
waves even when the mesh violates the ``rule-of-thumb" condition
(i.e., $6-10$ grid points must be used in a wave length).
The main difficulty of analyzing the Helmholtz type problems
is caused by the strong indefiniteness of the Helmholtz equation which in turn
makes it hard to establish stability estimates for its numerical
approximations. The loss of stability in the case of large wave numbers
results in an additional pollution error (besides the interpolation error)
in the global error bounds. Extensive research has been done to address the
question whether it is possible to reduce the pollution effect,
we refer the reader to Chapter 4 of \cite{ihlenburg98} and the references therein
for an detailed exposition in this direction.

Motivated by the success of \cite{fw08a,fw08b}, the primary
objective of this paper is to extend the work of \cite{fw08a,fw08b}
to the {\em local discontinuous Galerkin} (LDG) formulation,
which is known to be more ``physical" and flexible than the IPDG
formulation on designing DG schemes \cite{ABCM2002,CS1998b}.
As it is well-known now, the key step for constructing LDG methods
is to design the numerical fluxes. As soon as the numerical fluxes
are selected, for a large class of coercive elliptic and parabolic
second order problems, there is a general framework for carrying out
convergence analysis of LDG methods \cite{ABCM2002}. Unfortunately,
this general framework does not apply to the Helmholtz type
problems which is extremely noncoercive/indefinite for large
wave number $k$. Nevertheless,
when designing the numerical fluxes for our LDG methods, we borrow
the idea of \cite{ABCM2002} by establishing the connection between
our LDG methods and the IPDG methods of \cite{fw08a,fw08b} although
it turns out that the IPDG methods of \cite{fw08a,fw08b} do not have
exactly equivalent LDG formulations due to the non-standard penalty terms
used in \cite{fw08a,fw08b}. This then leads to the construction
of our first LDG method.  It is proved and numerically verified
that this LDG method is absolutely stable and optimally convergent
for the scalar variable. However, it is sub-optimal for the
vector/flux variable. To improve the approximation accuracy for the
vector/flux variable, we design another set of numerical fluxes which
result in the construction of our second LDG method. It is
proved that the second LDG method is also absolutely stable and
gives a better approximation for the vector/flux variable
than the first method. On the other hand, it is computationally more
expensive than the first LDG method, which is expected.

To analyze the proposed LDG methods, we take an opposite approach
to that advocated in \cite{ABCM2002}, that is, instead of converting
LDG methods to their ``equivalent" IPDG methods in the primal form,
we recast and treat our LDG methods as nonconforming mixed finite
element methods. To avoid using the standard techniques such as
Schatz argument (cf. \cite{Brenner_Scott08,dssb93}) or Babu{\v{s}}ka's
{\em inf-sup} condition argument \cite{ib95a} to derive error
estimates (and to prove stability), both approaches would certainly
lead to  stringent mesh constraints, our main idea is to
establish a {\em generalized inf-sup} condition, which holds
without any mesh constraint, for each LDG method. The generalized
{\em inf-sup} conditions then immediately infer the desired absolute
stability of the proposed LDG methods. In return, the stability results
not only guarantee the well-posedness of the LDG methods but also
play a crucial role in the derivation of the (optimal) error estimates.

It should be pointed out that a lot of work has recently been done
on developing DG methods using piecewise plane wave functions,
oppose to simpler piecewise polynomial functions as done in this paper,
for the Helmholtz type problems. However, to the best of our
knowledge, none of these plane wave DG method is proved to be
absolutely stable with respect to wave number $k$ and mesh size $h$.
We refer the reader to \cite{Griesmaier_Monk10,Hiptmair_Perugia09,LHM09}
and the references therein for more discussions in this
direction. We also refer to \cite{fw08a,fw08b} for more
discussions and references on other discretization techniques
for the Helmholtz type problems.

This paper consists of four additional sections. In Section \ref{sec-2},
we introduce the notations used in this paper and present the
derivations of our two LDG methods. In Section \ref{sec-3},
we present a detailed stability analysis for both LDG methods.
The main task of the section is to prove a {\em generalized inf-sup} condition
for each proposed LDG method. Similar to \cite{fw08a,fw08b}, a
nonorthodox test function trick is the key to get the job done.
In Section \ref{sec-4}, a non-standard two-step error estimate
procedure is used to derive error estimates for the proposed
LDG methods. Once again, the stability estimates established in
Section \ref{sec-3} play a crucial role. Finally, Section \ref{sec-5}
contains some numerical experiments which are designed to
verify the theoretical error bounds proved in Section \ref{sec-4}
and to compare the performance of the proposed two LDG methods.

\section{Formulation of local discontinuous Galerkin methods}\label{sec-2}
The standard space, norm and inner product notation
are adopted in this paper. Their definitions can be
found in \cite{ABCM2002,Brenner_Scott08,CKS00,Riviere08}.
In particular, $(\cdot,\cdot)_Q$ and $\langle \cdot,\cdot\rangle_\Sigma$
for $\Sigma\subset \pa Q$ denote the $L^2$-inner product
on {\em complex-valued} $L^2(Q)$ and $L^2(\Sigma)$
spaces, respectively. $(\cdot,\cdot):=(\cdot,\cdot)_\Ome$
and $\langle \cdot,\cdot\rangle:=\langle \cdot,\cdot\rangle_{\pa\Ome}$.
Throughout the paper, $C$ is used to denote a generic positive constant
which is independent of $h$ and $k$. We also use the shorthand
notation $A\lesssim B$ and $B\gtrsim A$ for the
inequality $A\leq C B$ and $B\geq CA$. $A\simeq B$ is a shorthand
notation for the statement $A\lesssim B$ and $B\lesssim A$.

Assume that $\Ome \subset \bR^d\, (d=2,3)$ is a bounded and
{\em strictly star-shaped} domain with respect to a point $x_\Ome\in \Ome$.
We now recall the definition of star-shaped domains.
\begin{definition}\label{def1}
$Q\subset \bR^d$ is said to be a {\em star-shaped} domain with respect
to $x_Q\in Q$ if there exists a nonnegative constant $c_Q$ such that
\begin{equation}\label{estar}
(x-x_Q)\cdot \vec{n}_Q\ge c_Q \qquad \forall x\in\pa Q.
\end{equation}
$Q\subset \bR^d$ is said to be {\em strictly star-shaped} if $c_Q$ is positive.
\end{definition}

Let $\cT_{h}$ be a family of partitions of $\Omega $ parameterized
by $h>0$. For any triangle/tetrahedron $K\in \cT_{h}$, we define
$h_{K}:=\mbox{diam}(K)$ and $h:=\max_{K\in \cT_h} h_K$.
Similarly, for each edge/face $e$ of $K\in \cT_{h}$,
define $h_{e}:=\mbox{diam}(e)$. We assume that the elements of $\cT_{h}$
satisfy the minimal angle condition. We also define
\begin{align*}
\mathcal{E}_{h}^{I}&:=\mbox{ set of all interior edges/faces of $\cT_h$},
\\
\mathcal{E}_{h}^{B}&:=\mbox{ set of all boundary edges/faces of $\cT_h$ on
$\Ga=\pa\Ome$},\\
\mathcal{E}_{h}&:=\cE_h^I\cup \cE_h^B.
\end{align*}
Let $e$ be an interior edge shared by two elements $K_{1}$ and $K_{2}$
whose unit outward normal vectors are denoted by $\vec{n}_{1}$
and $\vec{n}_{2}$. For a scalar function $v$, let
$v_i=v|_{\partial K_i}$, and define
\begin{equation*}
\left\{ v\right\} =\frac{1}{2}( v_{1}+v_{2}),
\quad [v] =v_K-v_{K^\prime},
\quad [[ v]] =v_{1} \vec{n}_{1}+v_{2}\vec{n}_{2}
\quad \text{on } e\in \cE_h^I,
\end{equation*}
where K is $K_{1}$ or $K_{2}$, whichever has the bigger global labeling
and $K^{\prime }$ is the other.
For a vector field $\vec{v}$, let $\vec{v}_i=\vec{v}|_{\partial K_i}$
and define
\begin{equation*}
\{\vec{v}\} =\frac{1}{2}(\vec{v}_{1}+\vec{v}_{2}),
\quad [\vec{v}] =\vec{v}_K-\vec{v}_{K^{\prime }},
\quad [[\vec{v}]]=\vec{v}_1\cdot \vec{n}_1+\vec{v}_2\cdot \vec{n}_{2}
\quad \text{on }e\in \cE_h^I.
\end{equation*}

As it is well-known now (cf. \cite{ABCM2002}) that the first step
for formulating an LDG method is to rewrite the given PDE as a
first order system by introducing an auxiliary variable. For
the Helmholtz problem \eqref{helm-1}--\eqref{helm-2} we have
\begin{alignat}{2} \label{e2.2}
\sig &=\nabla u &&\qquad \text{in }\Omega, \\
-\Div\sig-k^{2}u &=f &&\qquad \text{in }\Omega, \label{e2.3}\\
\frac{\pa u}{\pa \vec{n}_\Ome}+\i ku &=g &&\qquad \text{on } \Ga,
\end{alignat}
Clearly, the vector-valued function (often called the flux variable)
$\sig$ is the auxiliary variable.

Then, multiplying \eqref{e2.2} and \eqref{e2.3} by
test functions $\overline{\ta}$ and $\overline{v}$, respectively, and
integrating both equations over an element $K\in \cT_h$ yields
\begin{align}
\int_K \sig\cdot \overline{\ta}\,dx &=-\int_{K} u\, \Div\overline{\ta}\,dx
+\int_{\pa K} u\,\vec{n}_K\cdot \overline{\ta}\,ds, \\
\int_K \sig\cdot \nabla \overline{v}\,dx - k^2\int_K u \overline{v}\,dx
&=\int_K f \overline{v}\,dx +\int_{\pa K} \sig\cdot \vec{n}_{K} \overline{v}\,ds,
\end{align}
where $\vec{n}_K$ denotes the unit outward normal vector to $\pa K$.
The above equations form the weak formulation one uses
to define LDG methods for the Helmholtz problem
\eqref{helm-1}--\eqref{helm-2}.

Next, we define LDG spaces as follows
\begin{align*}
V_h &:=\{v\in L^2(\Ome);\, \re(v)|_K, \im(v)\in P_r(K)\,\,\, \forall K\in \cT_h\},\\
\Sig_h &:=\{\ta\in (L^2(\Ome))^d;\,\, \re(\ta)|_K, \im(\ta)\in (P_\ell(K))^d\,\,\,
\forall K\in \cT_h\},
\end{align*}
where $P_r(K)\, (r\geq 1)$ stands for the set of all polynomials
of degree less than or equal to $r$ on $K$.

Finally, we are ready to define the following general LDG
formulation: Find $(u_h,\sig_h)\in V_h\times \Sig_h$ such that
for all $K\in \Gamma_h$ there hold
\begin{align}\label{sec2:LDGa}
\int_K \sig_h\cdot \overline{\ta}_h\,dx
&=-\int_{K} u_{h}\,\Div \overline{\ta}_h\,dx
+\int_{\pa K} \hat{u}_K\,\vec{n}_K\cdot \overline{\ta}_h\,ds, \\
\int_K\sig_h\cdot \nabla \overline{v}_h\,dx - k^2\int_K u_h \overline{v}_h\,dx
&=\int_K f \overline{v}_h\,dx
+\int_{\pa K}\hat{\sig}_K\cdot \vec{n}_{K} \overline{v}_h\,ds
\label{sec2:LDGb}
\end{align}
for any $(v_h,\ta_h)\in V_h\times \Sig_h$. Where the quantities
$\hat{u}_K$ and $\hat{\sig}_K$, which
are called numerical fluxes,  are respectively approximations
to $\sig=\nabla u$ and $u$ on the boundary $\pa K$
of $K$. As it is well-known now that the most important issue
for all LDG methods is how to choose the numerical fluxes.
The different choices of the numerical fluxes obviously lead
to different LDG methods. It is easy to understand that
these numerical fluxes must be chosen carefully in order to
ensure the stability and accuracy of the resulted LDG methods.

In this paper, we shall only consider the {\em linear element} case
(i.e., $r=\ell=1$) and propose two sets of numerical fluxes
$(\hat{u}_K, \hat{\sig}_K)$, which lead to two LDG methods.
Our choices of numerical fluxes are inspired by the interior
penalty discontinuous Galerkin (IPDG) methods proposed by
Feng and Wu \cite{fw08a} and are identified with the help of the
unified DG framework of \cite{ABCM2002} which bridges the primal
DG formulations (e.g. IPDG methods) and the flux DG
formulations (i.e., LDG methods).

\begin{enumerate}
\item {\em LDG method \#1}: Set
\begin{alignat*}{3}
&\hat{\sig}_K=\{ \nabla_h u_h\} -\i \beta [[u_h]],\quad
&&\hat{u}_K=\{u_h\}+\i \delta [[\nabla_h u_h]]
&&\quad\mbox{on } e\in \cE_h^I,\\
&\hat{\sig}_K=-\i ku_h\vec{n}_K +g\vec{n}_K,\quad
&&\hat{u}_K=u_h &&\quad\mbox{on } e\in \cE_h^B.
\end{alignat*}

\item {\em LDG method \#2}: Set
\begin{alignat*}{3}\
&\hat{\sig}_K=\{\sig_h\} -\i \beta [[u_h]],\quad
&&\hat{u}_K=\{u_h\} +\i\delta [[ \sig_h]]
&&\quad\mbox{on } e\in \cE_h^I,\\
&\hat{\sig}_K=-\i ku_h\vec{n}_K +g\vec{n}_K,\quad
&&\hat{u}_K=u_h &&\quad\mbox{on } e\in \cE_h^B.
\end{alignat*}
\end{enumerate}
Where $\beta$ and $\delta$ are positive constants to be specified later
and $\nabla_h$ denotes the piecewisely defined gradient operator
over $\cT_h$, that is, $\nabla_h|_K=\nabla|_K\,\,\forall K\in \cT_h$.

For the reader's convenience, we now briefly sketch the derivation
of the primal DG formulation corresponding to our LDG method \#1 by
adapting the derivation given in the general framework of \cite{ABCM2002}.

Substituting the numerical fluxes of LDG method \#1 into \eqref{sec2:LDGa}
and \eqref{sec2:LDGb}, summing the resulting
equations over all element $K\in\cT_h$ and using the
following integration by parts identity
\begin{align*}
(u_h,\Div\ta_h)_\Ome &=-(\nabla_h u_h,\ta_h)_{\Ome}
+\sum_{e\in \cE_h^I}\Bigl( \Langle \{u_h\},[[\ta_h]]\Rangle_e \\
&\hskip 1.1in
+\Langle [[u_h]],\{\ta_h\}\Rangle_e \Bigr)
+\Langle u_h,\vec{n}_K\cdot \ta_h\Rangle_{\Ga},
\end{align*}
we get
\begin{align}\label{e2.9a}
&(\sig_h,\ta_h)_{\Ome}-(\nabla_h u_h,\ta_h)_{\Ome} \\
&\hskip 0.6in
-\sum_{e\in \cE_h^I} \Bigl( \i\delta\Langle[[\nabla_h u_h]],[[\ta_h]]\Rangle_{e}
-\Langle [[u_h]],\{\ta_h\}\Rangle _{e}\Bigr)=0, \nonumber \\
&(\sig_h,\nabla_h v_h)_{\Ome}-k^{2}(u_h,v_h)_{\Ome}
-\sum_{e\in \cE_h^I} \Langle\{\nabla_h u_h\}
-\i \beta [[ u_h]],[[v_h]]\Rangle _{e}  \label{e2.9b}\\
&\hskip 1.6in
+\i k\Langle u_h,v_h\Rangle_{\Ga}=(f,v_h)_\Ome
+\Langle g,v_h\Rangle_{\Ga}. \nonumber
\end{align}

Setting $\ta=\nab v_h$ in \eqref{e2.9a} and subtracting the resulting
equation from \eqref{e2.9b} then leads to the following formulation:
\begin{equation} \label{e2.10}
\mathcal{A}_h(u_h,v_h)-k^2(u_h,v_h)_\Ome =F(v_h)\qquad \forall v_h\in V_h,
\end{equation}%
where
\begin{align*}
\mathcal{A}_h( u_h,v_h)
&:=(\nabla_h u_h,\nabla_h v_h)_{\Ome}+\i k\Langle u_h,v_h\Rangle_{\Ga}\\
&\qquad
+\sum_{e\in \cE_h^I}\i\Bigl(\delta\Langle [[\nabla_h u_h]],[[\nabla_h v_h]]\Rangle_e
+\beta \Langle [[u_h]],[[v_h]]\Rangle_e\Bigr) \\
&\qquad
-\sum_{e\in \cE_h^I} \Bigl(\Langle [[u_h]],\{\nabla_h v_h\}\Rangle_e
+\Langle \{ \nabla_h u_h\},[[v_h]]\Rangle_e \Bigr). \\
F(v_h)&:= (f,v_h)_\Ome +\Langle g,v_h\Rangle_{\Ga}.
\end{align*}

Hence, \eqref{e2.10} is the corresponding primal (i.e., IPDG) formulation
of our LDG method \#1. Comparing \eqref{e2.10} with the IPDG formulation
of \cite{fw08a}, we see that all the interior penalty terms of
\eqref{e2.10} also appear in the IPDG formulation of \cite{fw08a}.
In fact, it is exactly by reversing the order of the above derivation
that leads to the discovery of the numerical fluxes of LDG method \#1.

\section{Stability analysis}\label{sec-3}
From their constructions, it is easy to see that both LDG methods proposed
in the previous section are consistent schemes for the Helmholtz
problem \eqref{helm-1}--\eqref{helm-2}.  For coercive elliptic
(and parabolic) problems, the stability of such a numerical
scheme can be proved easily as demonstrated in \cite{ABCM2002}
(the same statement is true for their corresponding PDE
stability analysis).  However, the Helmholtz
problem \eqref{helm-1}--\eqref{helm-2} is an indefinite
problem and it becomes notoriously non-coercive for large
wave number $k$. Deriving its stability estimates (i.e.,
a priori estimates of its PDE solution), particularly
wave-number-dependent estimates, has been proved not to be an easy
job (cf. \cite{Cummings01,Cummings_Feng06,hetmaniuk07,melenk95}
and the references therein). Numerically, such a quest has
been known to be even harder because of the lower order of the smoothness
and the inflexibility of (piecewise) approximation functions
(cf. \cite{Cummings01,ib95a,melenk95}
and the references therein). The stability of the numerical
methods in all the above quoted references was proved under some
very restrictive mesh constraints. An open question was then
raised by Zienkiewicz \cite{zienkiewicz00} which asks
whether it is possible to construct an absolutely stable
(and optimally convergent) numerical method (i.e., no restriction
on the mesh size $h$ and the wave number $k$) for the Helmholtz equation.
Almost a decade later Feng and Wu \cite{fw08a,fw08b} were able to
design for the first time such numerical methods, which happen to
belong to the IPDG family, for the Helmholtz problem
\eqref{helm-1}--\eqref{helm-2}.

The goal of this section is to show that in the case of the
linear element (i.e., $r=\ell=1$) the LDG method \#1 and \#2
proposed in the previous section have the same stability
property as the IPDG methods of \cite{fw08a,fw08b} do, that is,
the LDG method \#1 and \#2 are absolutely stable
for all mesh size $h>0$ and all wave number $k>0$
without imposing any constraint on them. To establish this result,
unlike the approach used and advocated in the general framework
of \cite{ABCM2002} which converts an LDG method into its
equivalent primal method and then analyzes the latter using
the standard finite element Galerkin techniques,
we shall directly fit the LDG method \#1 and \#2 into the
(nonconforming) mixed method framework as done in \cite{ib95a}
and adapting the mixed method techniques to prove the desired
stability result. It is well-known that the key ingredient
for the mixed method approach is to establish the {\em inf-sup}
or Babu{\v{s}}ka-Brezzi condition for the (augmented)
sesquilinear forms for each method.  However, we are
not able to prove such an {\em inf-sup} condition for either
method without imposing mesh constraints (which we believe
is not possible). To overcome the difficulty, our main idea
is to prove a generalized
(and weaker) {\em inf-sup} condition which holds for all
$h,k>0$. It turns out that this generalized {\em inf-sup}
condition is sufficient for us to establish the desired
absolute stability for the LDG method \#1 and \#2.
To prove the generalized {\em inf-sup} condition, the key
technique we use is a special test function technique
which was first introduced and developed in \cite{fw08a}.

\subsection{Absolute stability of LDG method \#1} \label{sec-3.1}

We first recast the LDG method \#1 as the following nonconforming
mixed method, which is
easily obtained by adding \eqref{e2.9a} and \eqref{e2.9b}.
Find $(u_h,\sig_h)\in V_h\times \Sig_h$ such that
\begin{equation}\label{e3.1}
A_h(u_h,\sig_h;v_h,\ta_h) =F(v_h,\ta_h) \qquad \forall (v_h,\ta_h)\in
V_h\times\Sig_h,
\end{equation}
where
\begin{align}\label{e3.2}
&A_h(w_h,\ch_h;v_h,\ta_h)
=(\ch_h,\nabla_h v_h)_\Ome - k^2(w_h,v_h)_\Ome
+\i k\Langle w_h,v_h\Rangle_\Ga \\
&\hskip 1.4in
-\sum_{e\in \cE_h^I}\Langle \{\nabla_h w_h\}-\i \beta [[w_h]],[[v_h]] \Rangle_e
\nonumber \\
&\hskip 1.4in
-\sum_{e\in \cE_h^I} \Bigl( \i\delta\Langle [[\nabla_h w_h]],[[\ta_h]]\Rangle_e
-\Langle [[w_h]],\{\ta_h\} \Rangle_e \Bigr) \nonumber \\
&\hskip 1.4in
+(\ch_h,\ta_h)_\Ome -(\nabla_h w_h,\ta_h)_\Ome. \nonumber \\
&F(v_h,\ta_h):= (f,v_h)_\Ome +\Langle g,v_h\Rangle_{\Ga}.\label{e3.3}
\end{align}

\subsubsection{\bf A generalized {\em inf-sup} condition}\label{sec-3.1.2}
The goal of this subsection is to show that the
sesquilinear form $A_h$ defined in \eqref{e3.2}
satisfies a generalized {\em inf-sup} condition, which will play a vital
role for us to establish the absolute stability of
the LDG method \#1 in the next subsection.

\begin{proposition}\label{prop3.1}
There exists an $h$- and $k$-independent constant $c_1>0$ such that
for any $(w_h,\ch_h)\in V_h\times\Sig_h$
\begin{align}\label{e3.4a}
&\sup_{(v_h,\ta_h)\in V_h\times\Sig_h\atop v_h\neq 0}
\frac{\re A_h(w_h,\ch_h;v_h,\ta_h)}{\|v_h\|_{DG} } \\
&\hskip 0.8in
+\sup_{(v_h,\ta_h)\in V_h\times\Sig_h\atop v_h\neq 0}
\frac{\im A_h(w_h,\ch_h;v_h,\ta_h)}{\|v_h\|_{DG} }
\geq \frac{c_1}{\gamma_1} \|w_h\|_{DG},  \nonumber
\end{align}
where
\begin{align}\label{e3.4b}
&\gamma_1 := 1+k+\sqrt{\frac{\beta}{\delta}}
+\max_{e\in \cE_h^I} \Bigl( \frac{k^2+1}{\beta h_e}+\frac{1}{h_e^2}
+\frac{1}{\beta h_e^3} \Bigr), \\
&\|w_h\|_{DG}:= \Bigl( k^2\|w_h\|_{L^2(\Ome)}^2
+k^2 \|w_h\|_{L^2(\Gamma)}^2 + c_\Ome \|\nab_h w_h\|_{L^2(\Gamma)}^2
+ |w_h|_{1,h}^2 \Bigr)^{\frac12}, \label{e3.4c}\\
&|w_h|_{1,h}:=  \Bigl(\sum_{K\in\cT_h}  \|\nab w_h\|_{L^2(K)}^2 \Bigr)^\frac12 .\notag
\end{align}

\end{proposition}

\begin{proof}
The main idea of the proof is that for a fixed
$(w_h,\ch_h)\in V_h\times\Sig_h$ we need to pick up
two sets of special functions $(v_h,\ta_h)\in V_h\times\Sig_h$,
which, as expected, must depend on $(w_h,\ch_h)\in V_h\times\Sig_h$,
such that both quotients in \eqref{e3.4a} can be bounded
from below by $\|w_h\|_{DG}$. When that is done,
the {\em inf-sup} constant $c_1/\gamma_1$ will be revealed in the process.
Since the proof is very long, we divide it into four steps.

\medskip
{\em Step 1: Taking the first test function.}

We first choose the test function $(v_h,\ta_h)=(w_h,-\nabla_h w_h)$ to get
\begin{align*}
A_h(w_h,\ch_h;v_h,\ta_h)&=A_h(w_h,\ch_h;w_h,-\nabla_h w_h) \\
&=(\nabla_h w_h,\nabla_h w_h)_\Ome -k^2(w_h,w_h)_\Ome
+\i k\Langle w_h,w_h \Rangle_\Ga  \nonumber\\
&\qquad
+\sum_{e\in \cE_h^I} \Bigl( \i\delta\Langle [[\nabla_h w_h]],[[\nabla_h w_h]]\Rangle_e
-\Langle [[ w_h]],\{\nabla_h w_h\} \Rangle_e \Bigr) \nonumber\\
&\qquad
-\sum_{e\in \cE_h^I}\Langle \{\nabla_h w_h\}-\i\beta[[w_h]],[[w_h]] \Rangle_e.
\nonumber
\end{align*}
Taking the real and imaginary parts yields
\begin{align}\label{e3.5a}
&\re A_h(w_h,\ch_h;w_h,-\nabla_h w_h) \\
&\hskip 0.7in
=|w_h|_{1,h}^2 - k^2\|w_h\|_{L^2(\Ome)}^2
-2\re \sum_{e\in \cE_h^I} \Langle [[w_h]],\{\nabla_h w_h\}\Rangle_e,
\nonumber\\
&\im A_h(w_h,\ch_h;w_h,-\nabla_h w_h) \label{e3.5b}\\
&\hskip 0.7in
=\sum_{e\in \cE_h^I} \Bigl( \delta \|[[\nabla_h w_h]]\|_{L^2(e)}^2
+\beta \|[[w_h]]\|_{L^2(e)}^2 \Bigr)
+k\|w_h\|_{L^2(\Ga)}^2. \nonumber
\end{align}

\medskip
{\em Step 2: Taking the second test function.}

Inspired by the special test function technique of \cite{fw08a},
we now choose another test function
$(v_h,\ta_h)=(\al\cdot \nabla_h w_h,-\nabla_h w_h)$
with $\al:=x-x_\Ome$ (see Definition \ref{def1}) and
use the fact that $\nabla_h v_h=\nabla_h w_h$ to get
\begin{align*}
A_h(w_h,\ch_h;v_h,\ta_h)&= A_h(w_h,\ch_h;\al\cdot \nabla_h w_h, -\nabla_h w_h)
\\
&=(\nabla_h w_h,\nabla_h w_h)_\Ome -k^2(w_h,\al\cdot \nabla_h w_h)_\Ome
+\i k\Langle w_h,\al\cdot \nabla_h w_h\Rangle_\Ga\\
&\qquad
+\sum_{e\in \cE_h^I} \Bigl( \i\delta \Langle [[\nabla_h w_h]],[[\nabla_h w_h]]\Rangle_e
-\Langle [[w_h]],\{\nabla_h w_h\} \Rangle_e \Bigr) \\
&\qquad
-\sum_{e\in \cE_h^I} \Langle \{\nabla_h w_h\}
-\i \beta [[w_h]],[[\al\cdot \nabla_h w_h]]\Rangle_e.
\end{align*}

Taking the real part immediately gives (note that $v_h=\al\cdot \nabla_h w_h$)
\begin{align}\label{e3.6a}
&\re A_h(w_h,\ch_h;v_h,-\nabla_h w_h) \\
&\qquad
=|w_h|_{1,h}^2 - k^2(w_h,v_h)_\Ome -k\im \Langle w_h,v_h\Rangle_\Ga
-\sum_{e\in \cE_h^I} \beta \im \Langle [[w_h]],[[v_h]]\Rangle_e\notag\\
&\hskip1in
-\sum_{e\in \cE_h^I}\re\Big(\Langle [[w_h]],\{\nabla_h w_h\}\Rangle_e
+ \Langle \{\nabla_h w_h\},[[v_h]]\Rangle_e
 \Bigr).\nonumber
\end{align}

\medskip
{\em Step 3: Deriving an upper bound for $k^2\|w_h\|_{L^2(\Ome)}^2$.}

To bound \eqref{e3.5a} from below, we need to get an
upper bound for the term $k^2\|w_h\|_{L^2(\Ome)}^2$ on
the right hand side of \eqref{e3.5a}. This will be done
by carefully and judicially combining \eqref{e3.5b} and
\eqref{e3.6a} with some other differential identities, which
we now explain.

Using the integral identity
\begin{equation}\label{e3.6extra}
2\|w_h\|_{L^2(K)}^2=\int_{\pa K}\al\cdot \vec{n}_{K} |w_h|^2\, ds
-2\re(w_h,v_h)_K - (d-2)\|w_h\|_{L^2(K)}^2,
\end{equation}
and \eqref{e3.5a}, \eqref{e3.6a} we get
\begin{align}\label{e3.6b}
&2k^2\|w_h\|_{L^2(\Ome)}^2 = 2\re A_h(w_h,\ch_h,v_h,-\nabla_h w_h)\\
&\qquad
+(d-2)\re A_h(w_h,\ch_h;w_h,-\nabla_h w_h) \nonumber\\
&\qquad
+2k\im\Langle w_h,v_h\Rangle_\Ga
-2|w_h|_{1,h}^2 - (d-2)|w_h|_{1,h}^2 \nonumber\\
&\qquad
+2\re \sum_{e\in \cE_h^I} \Bigl( \Langle
[[w_h]],\{\nabla_h w_h\} \Rangle_e
+\Langle \{\nabla w_h\},[[v_h]]\Rangle_e \Bigr)  \nonumber \\
&\qquad
+ 2(d-2)\re \sum_{e\in \cE_h^I} \Langle [[w_h]],\{\nabla_h w_h\} \Rangle_e \nonumber \\
&\qquad
+\sum_{e\in \cE_h^I} \Bigl( 2\beta \im\Langle [[w_h]],[[v_h]]\Rangle_e
+k^2\sum_{K\in\cT_h} \int_{\pa K} \al\cdot \vec{n}_K |w_h|^2\, ds \Bigr).
\nonumber
\end{align}

By the elementary identity $|a|^2-|b|^2=\re (a+b)(\bar{a}-\bar{b})$
for any two complex numbers $a$ and $b$ we have
\begin{equation}\label{e3.6c}
\sum_{K\in \cT_h}\int_{\pa K}\al\cdot \vec{n}_{K} |w_h|^2\,ds
=2\re\sum_{e\in \cE_h^I}\Langle \al\cdot \vec{n}_e\{ w_h\},[w_h]\Rangle_e
+\Langle \al\cdot \vec{n}_\Ome,|w_h|^2\Rangle_\Ga.
\end{equation}
Next, using the local Rellich identity (see \cite[Lemma 4.1]{fw08a})
\begin{equation*}
(d-2)\|\nabla w_h\|_{L^2(K)}^2 +2 \re (\nabla w_h,\nabla v_h)_K
=\int_{\pa K} \al\cdot \vec{n}_K |\nabla v_h|^2\, ds
\end{equation*}
and the fact that $\nabla_h v_h=\nabla_h w_h$, we get
\begin{align}\label{e3.7}
d |w_h|_{1,h}^2 &= (d-2)|w_h|_{1,h}^2
+ 2\re\sum_{K\in \cT_h} (\nabla w_h,\nabla v_h)_K \\
&= \re\sum_{K\in \cT_h} \int_{\pa K} \al\cdot \vec{n}_K |\nabla v_h|^2\, ds
\nonumber \\
&=2\re \sum_{e\in \cE_h^I} \Langle \al\cdot \vec{n}_e \{\nabla_h w_h\},
[\nabla_h w_h]\Rangle_e
+\sum_{e\in \cE_h^B} \Langle \al \cdot \vec{n}_e,|\nabla_h w_h|^{2}\Rangle_e.
\nonumber
\end{align}
Substituting \eqref{e3.6c} and \eqref{e3.7} into \eqref{e3.6b} we get
\begin{align}\label{e3.8}
&2k^2\|w_h\|_{L^2(\Ome)}^2
=2\re A_h(w_h,\ch_h;v_h,-\nabla_h w_h) \\
&\qquad
+(d-2)\re A_h(w_h,\ch_h;w_h,-\nabla_h w_h) \nonumber\\
&\qquad
+2k^2 \re\sum_{e\in \cE_h^I}\Langle \al\cdot \vec{n}_e\{w_h\},[w_h]\Rangle_e
+k^2\Langle \al\cdot \vec{n}_\Ome,|w_h|^2\Rangle_\Ga \nonumber \\
&\qquad
+2k\im \Langle w_h,v_h\Rangle_\Ga
-\sum_{e\in \cE_h^B}\Langle \al\cdot \vec{n}_e,|\nabla_h w_h|^{2}\Rangle_e \nonumber \\
&\qquad
-2\re\sum_{e\in \cE_h^I} \Bigl(\Langle\al\cdot \vec{n}_e\{\nabla_h w_h\},
[\nabla_h w_h] \Rangle_e
-\Langle \{\nabla_h w_h\},[[v_h]]\Rangle_e \Bigr) \notag \\
&\qquad
+2(d-1)\re \sum_{e\in \cE_h^I}\Langle [[w_h]],\{\nabla_h w_h\}\Rangle_e
+2\im \sum_{e\in \cE_h^I}\beta \Langle [[w_h]],[[v_h]]\Rangle_e .
 \notag
\end{align}

To get an upper bound for $k^2\|w_h\|_{L^2(\Ome)}^2$, we need to
bound the terms on the right-hand side of \eqref{e3.8}, which
we now bound as follows.
\begin{align} \label{e3.9}
2k^2\re \sum_{e\in \cE_h^I}\Langle \al\cdot \vec{n}_{e}\{w_h\},[w_h]\Rangle_e
&\leq Ck^2\sum_{e\in \cE_h^I} h_e^{-\frac12}\|w_h\|_{L^2(K_e\cup K_{e}^\prime)}
\|[w_h]\|_{L^2(e)} \\
&\leq \frac{k^2}{2} \|w_h\|_{L^2(\Ome)}^2
+C\sum_{e\in \cE_h^I} \frac{k^2}{\beta h_e} \beta \|[w_h]\|_{L^2(e)}^2. \notag\\
k^2\Langle \al\cdot \vec{n}_\Ome,|w_h|^2\Rangle_\Ga
&\leq Ck^2\|w_h\|_{L^2(\Ga)}^2.  \label{e3.10}
\end{align}
It follows from the star-shaped assumption on $\Ome$ that
\begin{align} \label{e3.11}
2k\im\Langle w_h,v_h\Rangle_{\Ga}
&-\sum_{e\in \cE_h^B}\Langle \al\cdot \vec{n}_e,|\nabla_h w_h|^2\Rangle_e \\
& \leq Ck\sum_{e\in \cE_h^B}\|w_h\|_{L^2(e)} \|\nabla_h w_h\|_{L^2(e)}
-c_{\Ome}\sum_{e\in \cE_h^B} \|\nabla_h w_h\|_{L^2(e)}^2  \notag \\
& \leq Ck^2\|w_h\|_{L^2(\Ga)}^2
-\frac{c_{\Ome}}{2} \|\nabla_h w_h\|_{L^2(\Gamma)}^2.  \notag
\end{align}
By the trace inequality \cite{Brenner_Scott08}, we also have
\begin{align} \label{e3.13}
&2d\re\sum_{e\in \cE_h^I}\Langle [[w_h]],\{\nabla_h w_h\}\Rangle_e \\
&\hskip 0.5in
\lesssim 2d\sum_{e\in \cE_h^I} h_e^{-\frac12}
\sum_{K=K_e,K_e^\prime} \|\nabla_h w_h\|_{L^2(K)} \|[w_h]\|_{L^2(e)} \notag \\
&\hskip 0.5in
\leq\frac14 |w_h|_{1,h}^2 +C\sum_{e\in \cE_h^I}\frac{1}{\beta h_e}
\beta \|[w_h]\|_{L^2(e)}^2.  \notag
\end{align}
\begin{align} \label{e3.14}
& -2\re\sum_{e\in \cE_h^I} \Bigl( \Langle \al\cdot \vec{n}_e\{\nabla_h w_h\}, [\nabla_h w_h]\Rangle_e
- \Langle \{\nabla_h w_h\},[[v_h]]\Rangle_e \Bigr) \\
& \hskip0.8in
=2\re \sum_{e\in \cE_h^I} \left[ \sum_{j=1}^{d-1}
\int_e \Bigl((\al\cdot \ta_e^j) \{\nabla_h w_h\cdot \vec{n}_e\}\notag\right.  \\
& \hskip 1.8in \left.
-(\al\cdot \vec{n}_e) \{\nabla_h w_h\cdot \ta_e^j\} \Bigr)
\nabla_h[w_h]\cdot\ta_e^j \right] \notag \\
& \hskip0.8in
\lesssim \sum_{e\in \cE_h^I}\sum_{j=1}^{d-1}h_e^{-\frac12}
\sum_{K=K_e,K_e^{\prime }} \|\nabla_h w_h\|_{L^2(K)}
\|[\nabla_h w_h\cdot \ta_e^j]\|_{L^2(e)} \notag \\
& \hskip0.8in
\leq \frac14 |w_h|_{1,h}^2 + C\sum_{e \in \cE_h^I}\frac{1}{\beta h_e}
\sum_{j=1}^{d-1}\beta \|[\nabla_h w_h\cdot \ta_e^j]\|_{L^2(e)}^2.  \notag
\end{align}
By the definition of $v_h:=\al\cdot \nabla_h w_h$, we get
\begin{align} \label{e3.12}
&2\im\sum_{e\in \cE_h^I}\beta \Langle [[w_h]],[[v_h]]\Rangle_e
=2\im\sum_{e\in \cE_h^I}\beta \Langle [w_h],[v_h]\Rangle_e\\
&\quad
=2\im\sum_{e\in \cE_h^I}\beta \Langle [w_h],
\Bigl[(\al\cdot \vec{n}_e) \nabla_h w_h\cdot\vec{n}_e
+\sum_{j=1}^{d-1}(\al\cdot \ta_e^j) \nabla_h w_h\cdot \ta_e^j \Bigr]
\Rangle_e \nonumber \\
&\quad
 \leq C\sum_{e\in \cE_h^I}\beta \|[w_h]\|_{L^2(e)}
\|[\nabla_h w_h\cdot\vec{n}_e]\|_{L^2(e)} \notag \\
&\qquad
+C\sum_{e\in \cE_h^I}\beta \|[w_h]\|_{L^2(e)}
\sum_{j=1}^{d-1}\|[\nabla_h w_h\cdot \ta_e^j]\|_{L^2(e)} \notag \\
&\quad \leq C\sqrt{\frac{\beta}{\delta}}\sum_{e\in \cE_h^I}
\Bigl( \beta \|[w_h]\|_{L^2(e)}^2
+\delta \|[\nabla_h w_h\cdot\vec{n}_e]\|_{L^2(e)}^2
\Bigr)  \notag \\
&\qquad
+C\beta \sum_{e\in \cE_h^I} \Bigl( \|[w_h]\|_{L^2(e)}^2
+\sum_{j=1}^{d-1} \|[\nabla_h w_h\cdot \ta_e^j]\|_{L^2(e)}^2 \Bigr),  \notag
\end{align}
where $\{\ta_e^j\}_{j=1}^{d-1}$ denotes an orthogonal tangential frame
on the edge/face $e$, and we have used the decomposition
$\al=(\al\cdot\vec{n}_e) \vec{n}_e + \sum_{j=1}^{d-1}(\al\cdot \ta_e^j) \ta_e^j$.

Now substituting estimates \eqref{e3.9}--\eqref{e3.14} into
\eqref{e3.8} we obtain
\begin{align}\label{e3.15}
&2k^2\|w_h\|_{L^2(\Ome)}^2
\leq 2\re A_h(w_h,\ch_h;v_h,-\nabla_h w_h) \\
&\qquad
+(d-2)\re A_h(w_h,\ch_h;w_h,-\nabla_h w_h) \notag\\
&\qquad
+Ck^2\|w_h\|_{L^2(\Ga)}^2
-\frac{c_{\Ome}}{2} \|\nabla_h w_h\|_{L^2(\Gamma)}^2
+\frac{k^2}{2} \|w_h\|_{L^2(\Ome)}^2 \notag\\
&\qquad
+C\sum_{e\in \cE_h^I} \frac{k^2}{\beta h_e} \beta \|[w_h]\|_{L^2(e)}^2
-2\re\sum_{e\in \cE_h^I} \Langle [[w_h]],\{\nabla_h w_h\}\Rangle_e\notag \\
&\qquad
+\frac14 |w_h|_{1,h}^2
+C\sum_{e\in \cE_h^I} \frac{1}{\beta h_e}\sum_{j=1}^{d-1}
\beta \|[\nabla_h\cdot \ta_e^j]\|_{L^2(e)}^2 \notag \\
&\qquad
+\frac14 |w_h|_{1,h}^2 +C\sum_{e\in \cE_h^I}\frac{1}{\beta h_e}
\beta \|[w_h]\|_{L^2(e)}^2 \notag\\
&\qquad
+C\sum_{e\in \cE_h^I}\sqrt{\frac{\beta}{\delta}}
\Bigl( \beta \|[w_h]\|_{L^2(e)}^2
+\delta \|[\nab_h w_h\cdot \vec{n}_e]\|_{L^2(e)}^2 \Bigr) \notag\\
&\qquad
+C\sum_{e\in \cE_h^I} \Bigl( \beta \|[w_h]\|_{L^2(e)}^2
+\sum_{j=1}^{d-1}\beta \|[\nabla_h w_h\cdot \ta_e^j]\|_{L^2(e)}^2 \Bigr).\notag
\end{align}

On noting that \eqref{e3.5b} provides upper bounds for terms
$\|[\nabla_h w_h\cdot \vec{n}_e]\|_{L^2(e)}^2$,
$\|[w_h]\|_{L^2(e)}^2$ and $k^2 \|w_h\|_{L^2(\Ga)}^2$
in terms of $\im A_h(w_h,\ch_h;w_h,-\nabla_h w_h)$, using these
bounds in \eqref{e3.15} we get
\begin{align}\label{e3.16}
&2k^2\|w_h\|_{L^2(\Ome)}^2
+\frac{k^2}{2} \|w_h\|_{L^2(\Ga)}^2
+\frac{c_{\Ome}}{2} \|\nabla_h w_h\|_{L^2(\Gamma)}^2\\
&\qquad \leq 2 \re A_h(w_h,\ch_h;v_h,-\nabla_h w_h)
+(d-2)\re A_h(w_h,\ch_h;w_h,-\nabla_h w_h) \notag\\
&\qquad\qquad
+ \frac12 |w_h|_{1,h}^2 +\frac{k^2}{2} \|w_h\|_{L^2(\Ome)}^2
-2\re\sum_{e\in \cE_h^I} \Langle \{\nabla_h w_h\},[[w_h]]\Rangle_e \notag \\
&\qquad\qquad
+C\sum_{e\in \cE_h^I} \Bigl(\frac{1}{\beta h_e}+1\Bigr)\sum_{j=1}^{d-1}
\beta \|[\nabla_h w_h\cdot \ta_e^j]\|_{L^2(e)}^2  \notag \\
&\qquad\qquad
+ \cM_1 \im A_h(w_h,\ch_h;w_h,-\nabla_h w_h), \notag
\end{align}
where
\begin{align}\label{e3.17}
\cM_1:=C \Bigl(1+k+\sqrt{\frac{\beta}{\delta}}
+\max_{e\in \cE_h^I} \frac{k^2+1}{\beta h_e} \Bigr).
\end{align}

To bound the jumps of the tangential derivatives
$\|[\nabla_h w_h\cdot \ta_e^j]\|_{L^2(e)}^2$ in \eqref{e3.15}, we appeal to
the inverse inequality
\begin{equation}\label{e3.17extra}
\|[\nabla_h w_h\cdot \ta_e^j]\|_{L^2(e)}^2
\leq Ch_e^{-2}\|[w_h]\|_{L^2(e)}^2,
\end{equation}
and using \eqref{e3.5b} and \eqref{e3.5a} to get
\begin{align*}
&2k^2\|w_h\|_{L^2(\Ome)}^2
+\frac{k^2}{2} \|w_h\|_{L^2(\Ga)}^2
+\frac{c_{\Ome}}{2} \|\nabla_h w_h\|_{L^2(\Gamma)}^2 \\
&\qquad
\leq 2 \re A_h(w_h,\ch_h;v_h,-\nab_h w_h)
+(d-2)\re A_h(w_h,\ch_h;w_h,-\nab_h w_h) \\
&\hskip .8in
+\frac{k^2}{2} \|w_h\|_{L^2(\Ome)}^2
+\cM_2 \im A_h(w_h,\ch_h;w_h,-\nabla_h w_h) \\
& \hskip .8in
-2\re\sum_{e\in \cE_h^I} \Langle \{\nab_h w_h\},[[w_h]]\Rangle_e
+ \frac12 |w_h|_{1,h}^2  \\
&\qquad
\leq 2 \re A_h(w_h,\ch_h;v_h,-\nab_h w_h)
+(d-1)\re A_h(w_h,\ch_h;w_h,-\nab_h w_h) \\
&\hskip .8in
+ \frac{3k^2}{2} \|w_h\|_{L^2(\Ome)}^2
+\cM_2 \im A_h(w_h,\ch_h;w_h,-\nab_h w_h)
- \frac12 |w_h|_{1,h}^2,
\end{align*}
where
\begin{align}\label{e3.18}
\cM_2&=\cM_1+\max_{e\in \cE_h^I} \Bigl( \frac{C}{h_e^2} +
\frac{C}{\beta h_e^3} \Bigr) \\
&= C \Bigl(1+k+\sqrt{\frac{\beta}{\delta}}
+\max_{e\in \cE_h^I} \Bigl( \frac{k^2+1}{\beta h_e} +\frac{1}{h_e^2}
+\frac{1}{\beta h_e^3} \Bigr)\Bigr). \notag
\end{align}
Using the linearity of the sesquilinear form $A_h$ we have
\begin{align}\label{e3.18extra}
& k^2\|w_h\|_{L^2(\Ome)}^2
+k^2 \|w_h\|_{L^2(\Ga)}^2
+ c_{\Ome} \|\nabla_h w_h\|_{L^2(\Gamma)}^2 + |w_h|_{1,h}^2 \\
&\qquad
\leq 4 \re A_h(w_h,\ch_h;v_h,-\nab_h w_h)
+2(d-1)\re A_h(w_h,\ch_h;w_h,-\nab_h w_h) \notag \\
&\hspace{0.8in}
+2\cM_2 \im A_h(w_h,\ch_h;w_h,-\nab_h w_h) \notag \\
& \qquad
= \re A_h(w_h,\ch_h;\tilde{w}_h,-\nab_h w_h)
+ 2\cM_2 \im A_h(w_h,\ch_h;w_h,-\nab_h w_h). \notag
\end{align}
with $\tilde{w}_h=4v_h+2(d-1)w_h$.

\medskip
{\em Step 4: Finishing up.}

By the definition of $\|\cdot\|_{DG}$ in \eqref{e3.4c} and the fact that
$\nabla_h v_h=\nabla_h w_h$ we have
\begin{align} \label{e3.19b}
\|v_h\|_{DG}^2
&=k^2 \|\al \cdot\nab_h w_h\|_{L^2(\Ome)}^2
+k^2 \|\al \cdot\nab_h w_h\|_{L^2(\Gamma)}^2 \\
&\hskip 1.33in
+ c_\Ome \|\nab_h w_h\|_{L^2(\Gamma)}^2
+ |w_h|_{1,h}^2 \notag \\
&\leq C k^2\|\nab_h w_h\|_{L^2(\Ome)}^2
+ Ck^2 \|\nab_h w_h\|_{L^2(\Gamma)}^2 \notag \\
&\hskip 1.25in
+ c_{\Ome}\|\nab_h w_h\|_{L^2(\Gamma)}^2
+ |w_h|_{1,h}^2 \notag\\
&\leq C(1+k^2) \Bigl( |w_h|_{1,h}^2
+ c_{\Ome} \|\nab_h w_h\|_{L^2(\Gamma)}^2 \Bigr) \notag \\
&\leq C(1+k^2) \|w_h\|_{DG}^2. \notag
\end{align}
Thus, it follows from the triangle inequality that
\begin{align}\label{e3.19a}
\|\tilde{w}_h\|_{DG}
\leq 4\|v_h\|_{DG} + 2(d-1)\|w_h\|_{DG}
\leq C (1+k)\|w_h\|_{DG}.
\end{align}
Now from \eqref{e3.18extra} and \eqref{e3.19a} we have
\begin{align}\label{e3.20}
&\frac{\re A_h(w_h,\ch_h;\tilde{w}_h,-\nab_h w_h)}
{\|\tilde{w}_h\|_{DG} }
+ \frac{\im A_h(w_h,\ch_h;w_h,-\nab_h w_h)}{\|w_h\|_{DG} }\\
&\geq\frac{\re A_h(w_h,\ch_h;\tilde{w}_h,-\nab_h w_h)}
{C(1+k) \|w_h\|_{DG}}
+\frac{\im A_h(w_h,\ch_h;w_h,-\nab_h w_h)}{\|w_h\|_{DG} }\notag \\
&\geq \frac{1}{2\cM_2}\cdot
\frac{\re A_h(w_h,\ch_h;\tilde{w}_h,-\nab_h w_h)
      +2\cM_2\im A_h(w_h,\ch_h;w_h,-\nab_h w_h)}{\|w_h\|_{DG}}\notag\\
&\geq \frac{c_1}{\gamma_1}\|w_h\|_{DG} \notag
\end{align}
for some constant $c_1>0$ and $\gamma_1$ is defined by \eqref{e3.4b}.
Hence, \eqref{e3.4a} holds. The proof is complete.
\end{proof}

\begin{remark}
(a) We note that $\gamma_1$ depends on both $h$ and $k$.

(b) The generalized {\em inf-sup} condition is a weak estimate because
it does not provide a control for the variable $\ch_h$. As a comparison,
we recall that the standard {\em inf-sup} condition for the sesquilinear form
$A_h$ should be
\begin{align*}
&\sup_{(v_h,\tau_h)\in V_h\times\Sig_h}
\frac{\bigl| A_h(w_h,\ch_h;v_h,\ta_h)\bigr|}{\|(v_h,\ta_h)\|}
\geq c_1 \|(w_h,\ch_h)\| \quad\forall (w_h,\ch_h)\in V_h\times\Sig_h
\end{align*}
for some positive constant $c_1=c_1(k,\beta,\delta,\Ome)$. Where
\[
\|(w_h,\ch_h)\|:= \bigl( k^2 \|w_h\|_{L^2(\Ome)}^2 
+ \|\ch_h\|_{L^2(\Ome)}^2\bigr)^{\frac12}.
\]
However, the above standard {\em inf-sup} condition can be proved
only under the mesh constraint $h=O(k^{-2})$ and we believe that it does not
hold without a mesh constraint.
\end{remark}

\subsubsection{\bf Stability estimates} \label{sec-3.1.3}

The goal of this subsection is to establish the absolute stability
for the LDG method \#1 using the generalized {\em inf-sup} condition proved
in the previous subsection.

\begin{theorem}\label{ldg1_sta}
Let $(u_h, \sig_h)\in V^h\times\Sig_h$ solve \eqref{e3.1}.  Define
\begin{equation}\label{Mfg}
M(f,g) :=\|f\|_{\Ome} + \|g\|_{L^2(\Ga)}.
\end{equation}
Then there hold the following stability estimates:
\begin{align}\label{e3.22}
\|u_h\|_{DG}  &\lesssim \gamma_1 k^{-1}\, M(f,g).\\
\|\sig_h\|_{L^2(\Ome)} &\lesssim
\gamma_1  k^{-1}
\Bigl(1 + (\delta +k^{-1})\bigl(\max_{K\in \cT_h}h_K^{-1}\bigr) \Bigr)\,
M(f,g).
\label{e3.23}
\end{align}
\end{theorem}

\begin{proof}
By Schwarz inequality we have
\begin{align}\label{e3.23a}
|F(v_h,\ch_h)| &\leq \|f\|_{L^2(\Ome)} \|v_h\|_{L^2(\Ome)}
+ \|g\|_{L^2(\Ga)} \|v_h\|_{L^2(\Ga)} \\
&\leq C k^{-1}M(f,g)\, \bigl(k^2\|v_h\|_{L^2(\Ome)}^2
+k^2\|v_h\|_{L^2(\Ga)}^2\bigr)^{\frac12}.  \notag \\
&\leq C k^{-1}  M(f,g) {\|v_h\|_{DG} }. \notag
\end{align}

Let $(w_h,\ch_h)=(u_h,\sig_h)$ in \eqref{e3.4a}.
By equation \eqref{e3.1} and \eqref{e3.23a} we get
\begin{align*}
\frac{c_1}{\gamma_1} \|u_h\|_{DG}
&\leq \sup_{(v_h,\ta_h)\in V_h\times\Sig_h}
\frac{\re A_h(u_h,\sig_h;v_h,\ta_h)}{\|v_h\|_{DG} } \\
&\hskip 0.8in
+\sup_{(v_h,\ta_h)\in V_h\times\Sig_h}
\frac{\im A_h(u_h,\sig_h;v_h,\ta_h)}{\|v_h\|_{DG} }  \\
&= \sup_{(v_h,\ta_h)\in V_h\times\Sig_h}
\frac{\re F(v_h,\ta_h)}{\|v_h\|_{DG} }
+\sup_{(v_h,\ta_h)\in V_h\times\Sig_h}
\frac{\im F(v_h,\ta_h)}{\|v_h\|_{DG} } \\
&\leq 2 \sup_{(v_h,\ta_h)\in V_h\times\Sig_h}
\frac{\bigl|F(v_h,\ta_h)\bigr|}{\|v_h\|_{DG} } \notag\\
&\leq C k^{-1} M(f,g). \notag
\end{align*}
Hence \eqref{e3.22} holds.

To show \eqref{e3.23}, setting $(v_h,\ta_h)=(0,\sig_h)$ in \eqref{e3.1} and
using the trace and Schwarz inequalities yields
\begin{align*}
\|\sig_h\|_{L^2(\Ome)}^2 &= (\nabla_h u_h,\sig_h)_{\Ome}
+\sum_{e\in \cE_h^I} \Bigl( \i\delta\Langle[[\nabla_h u_h]],[[\sig_h]]\Rangle_{e}
-\Langle [[u_h]],\{\sig_h\}\Rangle _{e}\Bigr) \\
&\leq \|\nabla_h u_h\|^2_{L^2(\Ome)} +\frac14\|\sig_h\|^2_{L^2(\Ome)}
+C\sum_{e\in \cE_h^I} \sum_{K=K_e,K_e'} h_K^{-1}
\Bigl( \delta \|\nabla_h u_h\|_{L^2(K)} \\
&\hskip 2.2in
+\|u_h\|_{L^2(K)} \Bigr)\|\sig_h\|_{L^2(K)} \\
&\leq\|\nabla_h u_h\|^2_{L^2(\Ome)} +\frac12\|\sig_h\|^2_{L^2(\Ome)} \\
&\hskip 0.7in
+C\sum_{e\in \cE_h^I} \sum_{K=K_e,K_e'}
h_K^{-2} \Bigl( \delta^2 \|\nabla_h u_h\|^2_{L^2(K)}+\|u_h\|^2_{L^2(K)}\Bigr).
\end{align*}
Thus,
\begin{align*}
\|\sig_h \|^2_{L^2(\Ome)}
&\lesssim \Bigl(1+ \delta^2 \bigl(\max_{K\in \cT_h}h_K^{-2}\bigr) \Bigr)
|u_h|^2_{1,h} + \bigl(\max_{K\in \cT_h}h_K^{-2}\bigr) \|u_h\|^2_{L^2(\Ome)}.
\end{align*}
The desired estimate \eqref{e3.23} follows from combining
the above inequality with \eqref{e3.22}. The proof is complete.
\end{proof}

An immediate consequence of the stability estimates is the following
unique solvability theorem.

\begin{theorem}\label{ldg1_existence}
There exists a unique solution to the LDG method \eqref{e3.1} for all
$k,\, h, \, \delta,\, \beta>0$.
\end{theorem}

\begin{proof}
Since problem \eqref{e3.1} is equivalent to a linear system, hence,
it suffices to show the uniqueness. But the uniqueness follows
immediately from the stability estimates as the zero sources imply
that any solution must be a trivial solution.
\end{proof}

\subsection{Absolute stability of LDG method \#2}\label{sec-3.2}

In this subsection, we consider the LDG method \#2. By adding
\eqref{sec2:LDGa} and \eqref{sec2:LDGb} with the given numerical fluxes,
we then recast the LDG method \#2 as the following nonconforming mixed method:
Find $(u_h,\sig_h)\in V_h\times \Sig_h$ such that
\begin{equation}\label{e3.2.1}
B_h(u_h,\sig_h;v_h,\ta_h) =F(v_h,\ta_h) \qquad \forall (v_h,\ta_h)\in
V_h\times\Sig_h,
\end{equation}
where $F$ is defined in \eqref{e3.3} and
\begin{align}\label{e3.2.2}
&B_h(w_h,\ch_h;v_h,\ta_h)
=(\ch_h,\nabla_h v_h)_\Ome - k^2(w_h,v_h)_\Ome + \i k\Langle w_h,v_h\Rangle_\Ga \\
&\hskip 1.4in
-\sum_{e\in \cE_h^I}\Langle \{\ch_h\}-\i \beta [[w_h]],[[v_h]] \Rangle_e
\nonumber \\
&\hskip 1.4in
-\sum_{e\in \cE_h^I} \Bigl( \i\delta\Langle [[\ch_h]],[[\ta_h]]\Rangle_e
-\Langle [[w_h]],\{\ta_h\} \Rangle_e \Bigr) \nonumber \\
&\hskip 1.4in
+(\ch_h,\ta_h)_\Ome -(\nabla_h w_h,\ta_h)_\Ome. \nonumber
\end{align}

\subsubsection{\bf A generalized {\em inf-sup} condition}\label{sec-3.2.1}
The goal of this subsection is to show that the
sesquilinear form $B_h$ defined in \eqref{e3.2.2} for the LDG method \#2
satisfies another generalized {\em inf-sup} condition. To the end, we introduce
the following space notation:
\begin{align}\label{spaceSh}
\mathcal{S}_h:=\{(w_h,\ch_h)\in V_h\times\Sig_h; (w_h,\ch_h) \,
\text{satisfies \eqref{sec3.2:lem3}} \},
\end{align}
where
\begin{align}\label{sec3.2:lem3}
&(\ch_h,\ta_h)_{\Ome}-(\nab_h w_h,\ta_h)_{\Ome} \\
&\hskip 0.1in
-\sum_{e\in \cE_h^I} \Bigl( \i\delta\Langle[[\ch_h]],[[\ta_h]]\Rangle_{e}
-\Langle [[w_h]],\{\ta_h\}\Rangle _{e}\Bigr)=0
\quad\forall (v_h,\ta_h)\in V_h\times\Sig_h. \nonumber
\end{align}
We note that it is easy to check that $(w_h,\ch_h)\in\mathcal{S}_h$
implies that it satisfies \eqref{sec2:LDGa} with $\hat{u}_K$ being defined
by the LDG method \#2.

\begin{lemma}
For any $(w_h,\ch_h)\in\mathcal{S}_h$, there holds the following estimates:
\begin{align} \label{sec3.2:lem1}
&|w_h|_{1,h} \leq \sqrt{\frac{17}{16}}\|\ch_h\|_{L^2(\Ome)} \\
&\hskip 0.6in
+ C\Bigl(\sum_{e \in \cE_h^I} \Bigl( \frac{1}{h_e} \|[[ w_h]]\|_{L^2(e)}^2
+ \frac{\delta^2}{h_e} \|[[ \ch_h]]\|_{L^2(e)}^2 \Bigr) \Bigr)^{\frac12},
\notag \\
&\|\ch_h-\nab_h w_h\|_{L^2(\Ome)} \leq
C\Bigl(\sum_{e \in \cE_h^I} \Bigl( \frac{1}{h_e} \|[[ w_h]]\|_{L^2(e)}^2
+ \frac{\delta^2}{h_e} \|[[ \ch_h]]\|_{L^2(e)}^2 \Bigr) \Bigr)^{\frac12}.
\label{sec3.2:lem2}
\end{align}
\end{lemma}

\begin{proof}
On noting that $(w_h,\ch_h)$ satisfies \eqref{sec3.2:lem3}, setting
$\ta_h=\nab_h w_h$ in \eqref{sec3.2:lem3}, we get
\begin{align*}
&|w_h|_{1,h}^2 =\re(\ch_h,\nab_h w_h)_{\Ome}\\
& \hskip0.8in
-\re\sum_{e\in \cE_h^I} \Bigl( \i\delta\Langle[[\ch_h]],[[\nab_h w_h]]\Rangle_{e}
-\Langle [[w_h]],\{\nab_h w_h\}\Rangle _{e}\Bigr)\\
& \hskip0.4in
\leq \frac12 |w_h|_{1,h}^2+ \frac12\|\ch_h\|^2_{L^2(\Ome)} \\
& \hskip0.8in
+\sum_{e\in \cE_h^I}\delta h_e^{-\frac12}\sum_{K=K_e,K_e^{\prime }}
\| [[\ch_h]] \|_{L^2(e)} \| \nab_hw_h \|_{L^2(K)} \\
& \hskip0.8in
+\sum_{e\in \cE_h^I}h_e^{-\frac12}\sum_{K=K_e,K_e^{\prime }}
\| [[w_h]] \|_{L^2(e)} \| \nab_hw_h \|_{L^2(K)} \\
& \hskip0.4in
\leq \frac12 |w_h|_{1,h}^2+ \frac12\|\ch_h\|^2_{L^2(\Ome)}
+ \frac{1}{34} |w_h|_{1,h}^2 \\
& \hskip0.8in
+ C\sum_{e \in \cE_h^I} \Bigl( \frac{1}{h_e} \|[[ w_h]]\|_{L^2(e)}^2
+ \frac{\delta^2}{h_e} \|[[ \ch_h]]\|_{L^2(e)}^2 \Bigr).
\end{align*}
Therefore,
\begin{align*}
|w_h|_{1,h}^2 \leq \frac{17}{16}\|\ch_h\|^2_{L^2(\Ome)}
+ C\sum_{e \in \cE_h^I} \Bigl( \frac{1}{h_e} \|[[ w_h]]\|_{L^2(e)}^2
+ \frac{\delta^2}{h_e} \|[[ \ch_h]]\|_{L^2(e)}^2 \Bigr),
\end{align*}
which gives \eqref{sec3.2:lem1}.

The estimate \eqref{sec3.2:lem2} follows from the same derivation by setting
$\ta_h=\ch_h - \nab_h w_h$ in \eqref{sec3.2:lem3}. The proof is complete.
\end{proof}

We now are ready to state a generalized {\em inf-sup} condition
for the sesquilinear form $B_h$.

\begin{proposition}\label{prop3.2}
There exits constant $c_2>0$ such that for any $(w_h,\ch_h)\in \mathcal{S}_h$
there holds
\begin{align}\label{e3.2.4a}
&\sup_{(v_h,\ta_h)\in V_h\times\Sig_h}
\frac{\re B_h(w_h,\ch_h;v_h,\ta_h)}{\norme{(v_h,\ta_h)}_{DG} }
\\
&\hskip 0.7in
+\sup_{(v_h,\ta_h)\in V_h\times\Sig_h}
\frac{\im B_h(w_h,\ch_h;v_h,\ta_h)}{\norme{(v_h,\ta_h)}_{DG} }
\geq \frac{c_2}{\gamma_2} \norme{(w_h,\ch_h)}_{DG},  \nonumber
\end{align}
where
\begin{align}\label{e3.2.4b}
&\gamma_2:=k+\max_{e\in \cE_h^I}
\Bigl( \frac{k^2+1}{\beta h_e} + \frac{\beta+\delta}{h_e}
+\frac{\delta}{h_e^3} +\frac{1}{\beta h_e^3} \Bigr), \\
&\label{e3.2.4c}
\norme{(w_h,\ch_h)}_{DG} :=\Bigl( k^2\|w_h\|_{L^2(\Ome)}^2
+k^2 \|w_h\|_{L^2(\Ga)}^2 \\
& \hskip 1.9in
+\|\ch_h\|^2_{L^2(\Ome)}
+c_{\Ome}\|\nabla_h w_h\|_{L^2(\Ga)}^2\Bigr)^{\frac12}. \notag
\end{align}
\end{proposition}

\begin{proof}
Since the proof follows the same lines as that of Proposition \ref{prop3.1},
we shall only highlight the main steps and point out the differences.

\medskip
{\em Step 1: Taking the first test function.}

We first choose the test function $(v_h,\ta_h)=(w_h,\ch_h)$ to get
\begin{align*}
B_h(w_h,\ch_h;v_h,\ta_h)&=B_h(w_h,\ch_h;w_h,\ch_h) \\
&=(\ch_h,\ch_h)_\Ome -k^2(w_h,w_h)_\Ome
+\i k\Langle w_h,w_h \Rangle_\Ga  \\
&\qquad
-\sum_{e\in \cE_h^I}\Langle \{\ch_h\}-\i\beta[[w_h]],[[w_h]] \Rangle_e  \\
&\qquad
-\sum_{e\in \cE_h^I} \Bigl( \i\delta\Langle [[\ch_h]],[[\ch_h]]\Rangle_e
-\Langle [[ w_h]],\{\ch_h\} \Rangle_e \Bigr) \\
& \qquad
+(\ch_h,\nab_h w_h)_{\Ome} - (\nab_h w_h,\ch_h)_{\Ome}.
\end{align*}
Taking the real and imaginary parts yields
\begin{align}\label{e3.2.5a}
\re B_h(w_h,\ch_h;w_h,\ch_h)
&=\|\ch_h\|^2_{L^2(\Ome)} - k^2\|w_h\|_{L^2(\Ome)}^2,\\
\im B_h(w_h,\ch_h;w_h,\ch_h)
&=k\|w_h\|_{L^2(\Ga)}^2 \label{e3.2.5b1}\\
&\quad
+\sum_{e\in \cE_h^I} \Bigl( -\delta \|[[\ch_h]]\|_{L^2(e)}^2
+\beta \|[[w_h]]\|_{L^2(e)}^2 \Bigr) \nonumber \\
&\quad
+ 2\im (\ch_h,\nab_h w_h)_{\Ome}
+\sum_{e\in \cE_h^I} 2\im\Langle [[ w_h]],\{\ch_h\} \Rangle_e. \nonumber
\end{align}

On noting that $(w_h,\ch_h)$ satisfies \eqref{sec3.2:lem3}, setting
$\ta_h=\ch_h$ in \eqref{sec3.2:lem3}, we get
\begin{align*}
&\|\ch_h\|_{L^2(\Ome)}^2-(\nab_h w_h,\ch_h)_{\Ome} =
\sum_{e\in \cE_h^I} \Bigl( \i\delta\|[[\ch_h]]\|_{L^2(e)}^2
-\Langle [[w_h]],\{\ch_h\}\Rangle _{e}\Bigr),
\end{align*}
Taking the imaginary part yields
\begin{align*}
&\im(\ch_h, \nab_h w_h)_{\Ome}
+ \sum_{e\in \cE_h^I} \im\Langle [[w_h]],\{\ch_h\}\Rangle _{e} =
\sum_{e\in \cE_h^I}  \delta \|[[\ch_h]]\|_{L^2(e)}^2.
\end{align*}
Hence, \eqref{e3.2.5b1} becomes
\begin{align}
\im B_h(w_h,\ch_h;w_h,\ch_h)
&=k\|w_h\|_{L^2(\Ga)}^2 \label{e3.2.5b}\\
&\quad
+\sum_{e\in \cE_h^I} \Bigl( \delta \|[[\ch_h]]\|_{L^2(e)}^2
+\beta \|[[w_h]]\|_{L^2(e)}^2 \Bigr). \nonumber
\end{align}

\medskip
{\em Step 2: Taking the second test function.}

Next, we choose another test function
$(v_h,\ta_h)=(\al\cdot \nabla_h w_h,\ch_h)$, which is different
from the one used in the proof of Proposition \ref{prop3.1},
and use the fact that $\nabla_h v_h=\nabla_h w_h$ to get
\begin{align*}
B_h(w_h,\ch_h;v_h,\ta_h)&= B_h(w_h,\ch_h,\al\cdot \nabla_h w_h, \ch_h)
\\
&=(\ch_h,\ch_h)_\Ome -k^2(w_h,\al\cdot \nabla_h w_h)_\Ome
+\i k\Langle w_h,\al\cdot \nabla_h w_h\Rangle_\Ga\\
&\quad
+\sum_{e\in \cE_h^I} \i \beta \Langle [[w_h]],[[\al\cdot \nabla_h w_h]]\Rangle_e
- \sum_{e\in \cE_h^I}\Langle \{\ch_h\},[[\al\cdot \nabla_h w_h]] \Rangle_e \\
&\quad
-\sum_{e\in \cE_h^I} \i\delta \Langle [[\ch_h]], [[\ch_h]]\Rangle_e
+\sum_{e\in \cE_h^I} \Langle [[w_h]],\{\ch_h\} \Rangle_e \\
&\quad
+2\i \im (\ch_h,\nab_h w_h).
\end{align*}
Taking the real part immediately gives ($v_h=\al\cdot \nabla_h w_h$)
\begin{align}\label{e3.2.6a}
&\re B_h(w_h,\ch_h;v_h,\ch_h) \\
&\qquad
=\|\ch_h\|_{L^2(\Ome)}^2 - k^2 \re(w_h,v_h)_\Ome
-k\im \Langle w_h,v_h\Rangle_\Ga \notag \\
&\hskip 0.6in
+\sum_{e\in \cE_h^I} \Bigl( \re\langle \{\ch_h\},[[w_h]]-[[v_h]]\rangle_e
- \beta \im \Langle [[w_h]],[[v_h]]\Rangle_e \Bigr).\nonumber
\end{align}

\medskip
{\em Step 3: Deriving an upper bound for $k^2\|w_h\|_{L^2(\Ome)}^2$.}

To bound \eqref{e3.2.5a} from below, we again need to get an
upper bound for the term $k^2\|w_h\|_{L^2(\Ome)}^2$ on
the right hand side of \eqref{e3.2.5a}.

Using the integral identity \eqref{e3.6extra}, \eqref{e3.6c}, \eqref{e3.2.5a}
and \eqref{e3.2.6a}, we have
\begin{align}\label{e3.2.8}
& 2k^2\|w_h\|_{L^2(\Ome)}^2
=2\re B_h(w_h,\ch_h;v_h,\ch_h) \\
&\qquad\qquad
+(d-2)\re B_h(w_h,\ch_h;w_h,\ch_h)-d\|\ch_h\|_{L^2(\Ome)}^2  \notag\\
&\qquad\qquad
+2k\im \Langle w_h,v_h\Rangle_\Ga
+k^2\Langle \al\cdot \vec{n}_\Ome,|w_h|^2\Rangle_\Ga\nonumber \\
&\qquad\qquad
+2k^2 \re\sum_{e\in \cE_h^I}\Langle \al\cdot \vec{n}_e\{w_h\},[w_h]\Rangle_e
+2\re  \sum_{e\in \cE_h^I} \Langle \{\ch_h\}, [[v_h]] \Rangle_e  \notag \\
&\qquad\qquad
+2(d-1)\re \sum_{e\in \cE_h^I}\Langle \{\ch_h\}, [[w_h]] \Rangle_e
+2\im \sum_{e\in \cE_h^I}\beta \Langle [[w_h]],[[v_h]]\Rangle_e  \notag \\
&\qquad\qquad
+ d |w_h|_{1,h}^2
-\sum_{e\in \cE_h^B} \Langle \al \cdot \vec{n}_e,|\nabla_h w_h|^2\Rangle_e \notag \\
&\qquad\qquad
- 2\re \sum_{e\in \cE_h^I} \Langle \al\cdot \vec{n}_e \{\nabla_h w_h\}, [\nabla_h w_h]\Rangle_e\notag \\
&\qquad\qquad
\pm 2\re \sum_{e\in \cE_h^I} \Langle \al\cdot \vec{n}_e \{\ch_h\}, [\nabla_h w_h]\Rangle_e. \notag
\end{align}
We note that by \eqref{e3.7} the sum of the second and third lines to the
last is zero, and the contribution of the last line is obviously zero. These
terms are purposely added in order to get sharper upper bounds when they
are combined with the terms preceding them.

We now need to bound the terms on the right-hand side of \eqref{e3.2.8}.
Some of these have been obtained in the proof of the Proposition \ref{prop3.1},
and the others are derived as follows.

\begin{align}\label{e3.2.13}
2(d-1)\re\sum_{e\in \cE_h^I} &\Langle \{\ch_h\},[[w_h]]\Rangle_e \\
&\lesssim 2(d-1)\sum_{e\in \cE_h^I} h_e^{-\frac12}
\sum_{K=K_e,K_e^\prime} \|\ch_h\|_{L^2(K)} \|[w_h]\|_{L^2(e)} \notag\\
&\leq\frac{1}{16} \|\ch_h\|_{L^2(\Ome)}^2
+C\sum_{e\in \cE_h^I}\frac{1}{\beta h_e}\beta \|[w_h]\|_{L^2(e)}^2.  \notag
\end{align}
\begin{align} \label{e3.2.14}
&2\re\sum_{e\in \cE_h^I} \Langle \{\ch_h\},[[v_h]]\Rangle_e
- 2\re \sum_{e\in \cE_h^I} \Langle \al\cdot \vec{n}_e \{\ch_h\}, [\nabla_h w_h]\Rangle_e \\
& \hspace{0.8in}
= 2\re \sum_{e\in \cE_h^I} \left[ \sum_{j=1}^{d-1}
\int_e \Bigl((\al\cdot \ta_e^j) \{\ch_h\cdot \vec{n}_e\}\notag\right.  \\
& \hspace{1.8in}
\left. -(\al\cdot \vec{n}_e) \{\ch_h\cdot \ta_e^j\} \Bigr)
\nabla_h[w_h]\cdot\ta_e^j \right] \notag \\
& \hspace{0.8in}
\lesssim \sum_{e\in \cE_h^I}\sum_{j=1}^{d-1}h_e^{-\frac12}
\sum_{K=K_e,K_e^{\prime }} \|\ch_h\|_{L^2(K)}
\|[\nabla_h w_h\cdot \ta_e^j]\|_{L^2(e)} \notag \\
& \hspace{0.8in}
\leq \frac{1}{16} \|\ch_h\|_{L^2(\Ome)}^2 + C\sum_{e \in \cE_h^I}\frac{1}{\beta h_e}
\sum_{j=1}^{d-1}\beta \|[\nabla_h w_h\cdot \ta_e^j]\|_{L^2(e)}^2.  \notag
\end{align}
It follows from \eqref{sec3.2:lem1} and \eqref{sec3.2:lem2} that
\begin{align}
& \label{e3.2.14b}
2\re \sum_{e\in \cE_h^I}
\Langle \al\cdot \vec{n}_e \{\ch_h - \nabla_h w_h\}, [\nabla_h w_h]\Rangle_e \\
& \qquad \notag
\leq  \sum_{e\in \cE_h^I} h_e^{-1} \sum_{K=K_e,K_e^{\prime }}
 \|\ch_h-\nab_hw_h\|_{L^2(K)} \|\nabla_h w_h\|_{L^2(K)} \\
& \qquad \notag
\leq  C\max_{e\in \cE_h^I}h_e^{-2} \|\ch_h-\nab_hw_h\|_{L^2(\Ome)}^2
+ \frac{1}{17}\|\nabla_h w_h\|_{L^2(\Ome)} \\
& \qquad \notag
\leq
 C\sum_{e \in \cE_h^I} \Bigl( \frac{1}{h_e^3} \|[[ w_h]]\|_{L^2(e)}^2
+ \frac{\delta^2}{h_e^3} \|[[ \ch_h]]\|_{L^2(e)}^2 \Bigr) \\
& \qquad \quad \notag
+ \frac{1}{16}\|\ch_h\|^2_{L^2(\Ome)}
+ C\sum_{e \in \cE_h^I} \Bigl( \frac{1}{h_e} \|[[ w_h]]\|_{L^2(e)}^2
+ \frac{\delta^2}{h_e} \|[[ \ch_h]]\|_{L^2(e)}^2 \Bigr) \\
& \qquad \notag
\leq
\frac{1}{16}\|\ch_h\|^2_{L^2(\Ome)}
+ C\sum_{e \in \cE_h^I} \Bigl( \frac{1}{h_e^3} \|[[ w_h]]\|_{L^2(e)}^2
+ \frac{\delta^2}{h_e^3} \|[[ \ch_h]]\|_{L^2(e)}^2 \Bigr).
\end{align}
Similar to the derivation of \eqref{e3.12}, we have
\begin{align} \label{e3.2.14c}
&2\im\sum_{e\in \cE_h^I}\beta \Langle [[w_h]],[[v_h]]\Rangle_e
=2\im\sum_{e\in \cE_h^I}\beta \Langle [w_h],[\al\cdot \nab_hw_h]\Rangle_e\\
&\quad
 \leq C\sum_{e\in \cE_h^I}\beta h_e^{-\frac12} \sum_{K=K_e,K_e^{\prime }}
 \|[w_h]\|_{L^2(e)} \|\nabla_h w_h\|_{L^2(K)} \notag \\
&\quad
\leq C\sum_{e\in \cE_h^I} \beta^2h_e^{-1} \|[w_h]\|_{L^2(e)}^2
+  \frac{1}{17}\|\nabla_h w_h\|_{L^2(\Ome)}^2 \notag \\
&\quad
\leq \frac{1}{16}\|\ch_h\|^2_{L^2(\Ome)}
+ C\sum_{e \in \cE_h^I} \Bigl( \frac{\beta^2+1}{h_e} \|[[ w_h]]\|_{L^2(e)}^2
+ \frac{\delta^2}{h_e} \|[[ \ch_h]]\|_{L^2(e)}^2 \Bigr).  \notag
\end{align}

Now substituting estimates \eqref{e3.2.13}--\eqref{e3.2.14c},
together with \eqref{e3.9}, \eqref{e3.10} and \eqref{e3.11}, into
\eqref{e3.2.8} we obtain
\begin{align}\label{e3.2.15}
&2k^2\|w_h\|_{L^2(\Ome)}^2
\leq 2\re B_h(w_h,\ch_h;v_h,\ch_h) \\
&\qquad
+(d-2)\re B_h(w_h,\ch_h;w_h,\ch_h)
-d\|\ch_h\|^2_{L^2(\Ome)} \notag\\
&\qquad
+ Ck^2\|w_h\|_{L^2(\Ga)}^2
-\frac{c_{\Ome}}{2}\sum_{e\in \cE_h^B} \|\nabla_h w_h\|_{L^2(e)}^2
+\frac{k^2}{2} \|w_h\|_{L^2(\Ome)}^2 \notag\\
&\qquad
+C\sum_{e\in \cE_h^I} \frac{k^2}{\beta h_e} \beta \|[w_h]\|_{L^2(e)}^2
+\frac{1}{16} \|\ch_h\|^2_{L^2(\Ome)}
+C\sum_{e\in \cE_h^I}\frac{1}{\beta h_e} \beta \|[w_h]\|_{L^2(e)}^2  \notag \\
&\qquad
+\frac{17}{16}d\|\ch_h\|^2_{L^2(\Ome)}
+ Cd\sum_{e \in \cE_h^I} \Bigl( \frac{\beta^2+1}{h_e} \|[[ w_h]]\|_{L^2(e)}^2
+ \frac{\delta^2}{h_e} \|[[ \ch_h]]\|_{L^2(e)}^2 \Bigr) \notag \\
& \qquad
+ \frac{1}{16} \|\ch_h\|_{L^2(\Ome)}^2 + C\sum_{e \in \cE_h^I}\frac{1}{\beta h_e}
\sum_{j=1}^{d-1}\beta \|[\nabla_h w_h\cdot \ta_e^j]\|_{L^2(e)}^2 \notag \\
&\qquad
+\frac{1}{16}\|\ch_h\|^2_{L^2(\Ome)}
+ C\sum_{e \in \cE_h^I} \Bigl( \frac{1}{h_e^3} \|[[ w_h]]\|_{L^2(e)}^2
+ \frac{\delta^2}{h_e^3} \|[[ \ch_h]]\|_{L^2(e)}^2 \Bigr) \notag\\
&\qquad
+\frac{1}{16}\|\ch_h\|^2_{L^2(\Ome)}
+ C\sum_{e \in \cE_h^I} \Bigl( \frac{\beta^2+1}{h_e} \|[[ w_h]]\|_{L^2(e)}^2
+ \frac{\delta^2}{h_e} \|[[ \ch_h]]\|_{L^2(e)}^2 \Bigr).\notag
\end{align}

On noting that \eqref{e3.2.5b} provides upper bounds for terms
$\|[[\ch_h]]\|_{L^2(e)}^2$, $\|[w_h]\|_{L^2(e)}^2$ and $k^2 \|w_h\|_{L^2(\Ga)}^2$
in terms of $\im B_h(w_h,\ch_h;w_h,\ch_h)$,  and
the jumps of the tangential derivatives
$\|[\nabla_h w_h\cdot \ta_e^j]\|_{L^2(e)}^2$ in \eqref{e3.2.15}
can be bounded by the inverse inequality \eqref{e3.17extra}, we get
\begin{align}\label{e3.2.17}
&2k^2\|w_h\|_{L^2(\Ome)}^2 +\frac{k^2}{2}\|w_h\|_{L^2(\Ga)}^2
+\frac{c_{\Ome}}{2}\sum_{e\in \cE_h^B} \|\nabla_h w_h\|_{L^2(e)}^2 \\
&\quad \leq 2 \re B_h(w_h,\ch_h;v_h,\ch_h)
+(d-2)\re B_h(w_h,\ch_h;w_h,\ch_h) \notag \\
&\qquad
+\cM_3 \im B_h(w_h,\ch_h;w_h,\ch_h)
+\frac{k^2}{2}\|w_h\|_{L^2(\Ome)}^2
+\Bigr(\frac14+\frac{d}{16}\Bigl)\|\ch_h\|^2_{L^2(\Ome)} \notag \\
&\quad \leq 2 \re B_h(w_h,\ch_h;v_h,\ch_h)
+(d-1)\re B_h(w_h,\ch_h;w_h,\ch_h) \notag \\
&\qquad
+\cM_3 \im B_h(w_h,\ch_h;w_h,\ch_h)
+\frac{3k^2}{2}\|w_h\|_{L^2(\Ome)}^2
- \frac12\|\ch_h\|^2_{L^2(\Ome)}, \notag
\end{align}
where
\begin{align}\label{e3.2.18}
\cM_3=C\Bigl(k+\max_{e\in \cE_h^I}
\Bigl( \frac{k^2+1}{\beta h_e} + \frac{\beta+\delta}{h_e}
+\frac{\delta}{h_e^3} +\frac{1}{\beta h_e^3} \Bigr) \Bigr).
\end{align}

Using the linearity of the sesquilinear form $B_h$, we have
\begin{align*}
& k^2\|w_h\|_{L^2(\Ome)}^2
+k^2\|w_h\|_{L^2(\Ga)}^2
+\|\ch_h\|^2_{L^2(\Ome)}
+ c_{\Ome} \|\nabla_h w_h\|_{L^2(\Ga)}^2 \\
&\qquad
\leq 4 \re B_h(w_h,\ch_h;v_h,\ch_h) +2(d-1)\re B_h(w_h,\ch_h;w_h,\ch_h) \\
&\hskip0.8in
+2\cM_3 \im B_h(w_h,\ch_h;w_h,\ch_h) \\
& \qquad
= \re B_h(w_h,\ch_h;\tilde{w}_h,\ch_h)
+2\cM_3 \im B_h(w_h,\ch_h;w_h,\ch_h),
\end{align*}
where $\tilde{w}_h=4v_h + 2(d-1)w_h$ and $v_h=\al\cdot\nab_h w_h$.

\medskip
{\em Step 4: Finishing up.}

It follows from the inverse inequality \eqref{e3.17extra} that
\[
k^2\|v_h\|_{L^2(\Ome)}^2 + k^2\|v_h\|_{L^2(\Gamma)}^2
\leq Ck^2 \Bigl(\max_{K\in \cT_h} h_K^{-2} \|w_h\|_{L^2(\Ome)}^2
+ c_\Ome \|\nab_h w_h\|_{L^2(\Gamma)}^2 \Bigr).
\]
Then by the definition of the norm $\norme{(w_h,\ch_h)}_{DG}$ in
\eqref{e3.2.4c} and the fact that $\nabla_h v_h=\nabla_h w_h$, we get
\begin{align*}
\norme{(\tilde{w}_h,\ch_h)}_{DG}
&\leq 4\norme{(v_h,\ch_h)}_{DG} + 2(d-1)\norme{(w_h,\ch_h)}_{DG} \\
&\leq C\bigl(\max_{K\in \cT_h} h_K^{-1}+k+1\bigr) \norme{(w_h,\ch_h)}_{DG}.
\end{align*}
Therefore,
\begin{align}\label{e3.2.20}
&
\frac{\re B_h(w_h,\ch_h;\tilde{w}_h,\ch_h)}
{\norme{(\tilde{w}_h,\ch_h)}_{DG} }
+ \frac{\im B_h(w_h,\ch_h;w_h,\ch_h)}
{\norme{(w_h,\ch_h)}_{DG} }\\
& \quad
\geq \frac{\re B_h(w_h,\ch_h;\tilde{w}_h,\ch_h)}
{C\bigl(\max_{K\in \cT_h} h_K^{-1}+k+1 \bigr) \norme{(w_h,\ch_h)}_{DG} }
+ \frac{\im B_h(w_h,\ch_h;w_h,\ch_h)}
{\norme{(w_h,\ch_h)}_{DG} }\notag \\
&\quad
\geq \frac{1}{2\cM_3}\frac{\re B_h(w_h,\ch_h;\tilde{w}_h,\ch_h)
+\cM_3\im B_h(w_h,\ch_h;w_h,\ch_h)}
{\norme{(w_h,\ch_h)}_{DG} } \notag \\
& \quad
\geq \frac{c_2}{\gamma_2} \norme{(w_h,\ch_h)}_{DG} ,\notag
\end{align}
for some constant $c_2>0$ and $\gamma_2$ defined by \eqref{e3.2.4b}.
Hence, \eqref{e3.2.4a} holds and the proof is complete.
\end{proof}

\subsubsection{\bf Stability estimates} \label{sec-3.2.2}

The generalized {\em inf-sup} condition proved in the last subsection
immediately infers the following (absolute) stability and well-posedness
theorems for the LDG method \#2.

\begin{theorem}\label{ldg2_sta}
Let $(u_h, \sig_h)\in V^h\times\Sig_h$ solve \eqref{e3.2.1}. Then there holds
\begin{align}\label{e3.2.22}
\norme{(u_h,\sig_h)}_{DG}  \lesssim \gamma_2 k^{-1}\, M(f,g).
\end{align}
\end{theorem}

\begin{proof}
On noting that any solution $(u_h,\sig_h)$ of \eqref{e3.2.1} belongs
to the set $\mathcal{S}_h$, the desired estimate \eqref{ldg2_sta}
follows readily from \eqref{e3.2.4a} with $(w_h,\ch_h)=(u_h,\sig_h)$,
\eqref{e3.2.1} and \eqref{e3.23a}. The proof is complete.
\end{proof}

\begin{theorem}\label{ldg2_existence}
The LDG method \eqref{e3.2.1} has a unique solution
for all $k,\, h,\, \delta,\, \beta>0$.
\end{theorem}

Since the proof of the above theorem is a verbatim copy of
that of Theorem \ref{ldg1_existence}, we omit it.

\section{Error estimates} \label{sec-4}

The goal of this section is to derive error estimates for the LDG method
\#1 and \#2. Following the idea of \cite{fw08a}, this will be done in
two steps. First, we introduce an elliptic projection of the solution
$(u,\sig)$ using a corresponding coercive sesquilinear form of
$A_h$ (resp. $B_h$) and derive error bounds for the projection.
We note that the error analysis for the elliptic projections
has an independent interest in itself (cf. \cite{CCPS02}).
Second, we bound the error between the projection and the LDG solution
using the stability estimates obtained in Section \ref{sec-3}.
Since the error analysis for the two LDG methods are similar,
we shall give more details of the error analysis for the LDG method \#1
but shall be brief for the LDG method \#2. Throughout this section, we let
for $j=1,2$
\[
H^j(\cT_h)=\prod_{K\in\cT_h} H^j(K),\quad
h=\max_{K\in \cT_h} h_K\approx \max_{e\in \cE_h} h_e,\quad
\beta=\beta_0\,h^{-1},\quad \delta=\delta_0\,h
\]
for some positive constants $\beta_0$ and $\delta_0$.

\subsection{Error estimates for the LDG method \#1} \label{sec-4.1}

\subsubsection{\bf Elliptic projection and its error estimates} \label{sec-4.1.1}

For any $(w,\ch)\in H^2(\cT_h) \times H^1(\cT_h)^d$,
we define the elliptic projection
$(\tilde{w}_h,\tilde{\chi}_h)\in V_h\times \Sig_h$ of $(w,\ch)$ by
\begin{align}\label{e4.1}
a_h(\tilde{w}_h,\tilde{\chi}_h;v_h,\ta_h)=a_h(w,\ch;v_h,\ta_h)
\qquad\forall (v_h,\ta_h)\in  V_h\times \Sig_h,
\end{align}
where
\begin{align}\label{e4.2}
a_h(w_h,\ch_h; v_h,\ta_h) :&=A_h(w_h,\ch_h ;v_h,\ta_h) +k^2 (w_h,v_h)_\Ome \\
&=(\ch_h,\nabla_h v_h)_\Ome +\i k\Langle w_h,v_h\Rangle_\Ga \nonumber \\
&\qquad
-\sum_{e\in \cE_h^I}\Langle \{\nabla_h w_h\}-\i \beta_0 h^{-1} [[w_h]],[[v_h]] \Rangle_e
\nonumber \\
&\qquad
-\sum_{e\in \cE_h^I} \Bigl( \i\delta\Langle [[\nabla_h w_h]],[[\ta_h]]\Rangle_e
-\Langle [[w_h]],\{\ta_h\} \Rangle_e \Bigr) \nonumber \\
&\qquad
+ (\ch_h,\ta_h)_\Ome - (\nabla_h w_h, \ta_h)_\Ome. \nonumber
\end{align}

To derive error bounds for the above elliptic projection,
we first notice that
\begin{align*}
a_h(w_h,\ch_h; v_h, -\nab_h v_h)
&=(\nab_h w_h,\nabla_h v_h)_\Ome +\i k\Langle w_h,v_h\Rangle_\Ga  \\
&\qquad
-\sum_{e\in \cE_h^I}\Langle \{\nabla_h w_h\}-\i \beta_0 h^{-1} [[w_h]],[[v_h]]\Rangle_e
\nonumber \\
&\qquad
+\sum_{e\in \cE_h^I} \Bigl( \i\delta\Langle [[\nabla_h w_h]],[[\nab_h v_h]]\Rangle_e
-\Langle [[w_h]],\{\nab_h v_h\} \Rangle_e \Bigr) \nonumber\\
&=:\mathcal{A}_h(w_h,v_h). \nonumber
\end{align*}
As a result, $\tilde{w}_h\in V_h$ satisfies
\begin{align}\label{e4.2a}
\mathcal{A}_h(\tilde{w}_h,v_h) = \mathcal{A}_h(w,v_h)
\qquad\forall v_h\in V_h.
\end{align}

Moreover, since
\begin{align*}
a_h(w_h,\ch_h; 0, \ta_h)
&=(\ch_h,\ta_h)_\Ome - (\nabla_h w_h, \ta_h)_\Ome  \\
&\qquad
-\sum_{e\in \cE_h^I} \Bigl( \i\delta\Langle [[\nabla_h w_h]],[[\ta_h]]\Rangle_e
-\Langle [[w_h]],\{\ta_h\} \Rangle_e \Bigr),
\end{align*}
we have that $\tilde{\chi}_h\in \Sig_h$ satisfies
\begin{align}\label{e4.2b}
(\tilde{\ch}_h, \ta_h)_\Ome &= (\nab_h \tilde{w}_h, \ta_h)_\Ome \\
&\quad
+ \sum_{e\in \cE_h^I} \Bigl( \i\delta\Langle [[\nabla_h \tilde{w}_h - \nabla_h w]] ,
[[\ta_h]] \Rangle_e -\Langle [[\tilde{w}_h - w]], \{\ta_h\} \Rangle_e \Bigr)\nonumber \\
&\quad
+(\ch,\ta_h)_\Ome -(\nab_h w, \ta_h)_\Ome
\qquad\forall \ta_h\in \Sig_h. \nonumber
\end{align}

\begin{lemma}\label{lem4.1}
For any $w,v\in H^2(\cT_h)$, there exists a $k$- and $h$-independent
constant $C$ such that
\begin{align}\label{e4.3}
&|\mathcal{A}_h(w, v)| \leq C |||w||||_{1,h} |||v|||_{1,h}.
\end{align}
Moreover, for any $\epsilon\in (0,1)$, there exists a constant $c_\epsilon>0$
such that
\begin{align}\label{e4.4}
&\re \mathcal{A}_h(v_h, v_h) + (1-\epsilon+c_\epsilon)\im \mathcal{A}_h(v_h, v_h)
\geq (1-\epsilon) \|v_h\|_{1,h}^2,
\end{align}
where
\begin{align}\label{e4.3a}
\|w\|_{1,h} &:=  \Bigl( \|\nab_h w\|_{L^2(\Ome)}^2 +k \|w\|_{L^2(\Gamma)}^2 \\
&\qquad\quad
+ \sum_{e\in\cE_h^I} \Bigl(\beta \|[w]\|_{L^2(e)}^2
+\delta \|[[\nab_h w]]\|_{L^2(e)}^2 \Bigr) \Bigr)^\frac12, \nonumber \\
|||w|||_{1,h} &:=  \Bigl( \|w\|_{1,h}^2
+\sum_{e\in\cE_h^I} \beta^{-1}\|\{\nab_h w\cdot \mathbf{n}_e\}\|_{L^2(e)}^2 \Bigr)^\frac12.
\label{e4.3b}
\end{align}
\end{lemma}

Since the proof of the above lemma is elementary, we omit it.
We now recall the following stability estimate for $u$
(cf. \cite{Cummings_Feng06,fw08a}):
\[
\|u\|_{H^2(\Ome)}\lesssim (k^{-1} + k) M(f,g),
\]
which is needed to prove the next lemma and will be used several times
in the rest of this section.

\begin{proposition}\label{prop4.1}
Let $u\in H^2(\Ome)$ be the solution to problem
\eqref{helm-1}-\eqref{helm-2} and $\sig=\nab u$.
Let $(\tilde{u}_h,\tilde{\sig}_h)\in V_h\times \Sig_h$ denote
the elliptic projection of $(u,\sig)$ defined by \eqref{e4.1}.
Then there hold the following error estimates:
\begin{align}\label{e4.5a}
\|u-\tilde{u}_h\|_{1,h} +k^{\frac12} \|u-\tilde{u}_h\|_{L^2(\Gamma)}
&\lesssim (1+k h)^{\frac12} kh,\\
\|u-\tilde{u}_h\|_{L^2(\Ome)} &\lesssim (1+k h)kh^2,  \label{e4.5b}\\
\|\sig-\tilde{\sig}_h\|_{L^2(\Ome)} &\lesssim (1+k h)^{\frac12}kh.\label{e4.5c}
\end{align}
\end{proposition}

\begin{proof}
Since the proof of \eqref{e4.5a} and \eqref{e4.5b} is essentially
same as that of \cite[Lemma 5.2]{fw08a}, we omit it to save the space
and refer the reader to \cite{fw08a} for the details.

To show \eqref{e4.5c}, on noting that \eqref{e4.2b} and the identity
$(\sig,\ta_h)=(\nab_h u,\ta_h)$ imply
\begin{align}\label{e4.5d}
(\sig-\tilde{\sig}_h, \ta_h)_\Ome &= (\nab_h u-\nab_h\tilde{u}_h, \ta_h)_\Ome
- \sum_{e\in \cE_h^I} \Bigl( \i\delta_0 h\Langle [[\nabla_h (u-\tilde{u}_h)]]
[[\ta_h]] \Rangle_e \\
&\hskip 0.6in
-\Langle [[u-\tilde{u}_h]], \{\ta_h\} \Rangle_e \Bigr)
\qquad\forall \ta_h\in \Sig_h. \nonumber
\end{align}
For any $\ch_h\in \Sig_h$, we set $\ta_h=\ch_h-\tilde{\sig}_h$. Then by \eqref{e4.5d},
the trace inequality, Schwarz inequality we get
\begin{align*}
&\|\sig-\tilde{\sig}_h\|_{L^2(\Ome)}^2
=(\sig-\tilde{\sig}_h, \sig-\ch_h)_\Ome + (\sig-\tilde{\sig}_h, \ta_h)_\Ome \\
&\qquad
=(\sig-\tilde{\sig}_h, \sig-\ch_h)_\Ome + (\nab_h (u-\tilde{u}_h), \ta_h)_\Ome \nonumber\\
&\qquad\qquad
-\sum_{e\in \cE_h^I} \Bigl( \i\delta_0 h\Langle [[\nabla_h (u-\tilde{u}_h)]],
[[\ta_h]] \Rangle_e -\Langle [[u-\tilde{u}_h]], \{\ta_h\} \Rangle_e \Bigr) \nonumber \\
&\qquad
\leq \|\sig-\tilde{\sig}_h\|_{L^2(\Ome)} \|\sig-\ch_h\|_{L^2(\Ome)}
+ \|\nab_h (u-\tilde{u}_h)\|_{L^2(\Ome)} \|\ta_h\|_{L^2(\Ome)} \nonumber\\
&\qquad\qquad
+C\Bigl(\sum_{e\in \cE_h^I} \delta_0 h
\|[[\nabla_h (u-\tilde{u}_h)]]\|_{L^2(e)}^2
+\beta_0 h^{-1}  \|[u-\tilde{u}_h]\|_{L^2(e)}^2 \nonumber\\
&\qquad\qquad\qquad
+\varepsilon\sum_{e\in \cE_h^I} h \|\ta_h\|_{L^2(e)}^2 \Bigr) \nonumber  \\
&\qquad
\leq \frac14 \|\sig-\tilde{\sig}_h\|_{L^2(\Ome)}^2 + \frac14 \|\ta_h\|_{L^2(\Ome)}^2
+\|\sig-\ch_h\|_{L^2(\Ome)}^2 + C\|u-\tilde{u}_h\|_{1,h}^2 \nonumber\\
&\qquad
\leq \frac12 \|\sig-\tilde{\sig}_h\|_{L^2(\Ome)}^2 + \frac54\|\sig-\ch_h\|_{L^2(\Ome)}^2
+ C\|u-\tilde{u}_h\|_{1,h}^2. \nonumber
\end{align*}
Hence, it follows from the above inequality, \eqref{e4.5a}, and the
polynomial approximation theory (cf. \cite{Brenner_Scott08}) that
\begin{align*}
\|\sig-\tilde{\sig}_h\|_{L^2(\Ome)}
&\leq C \|u-\tilde{u}_h\|_{1,h}+ 2\inf_{\ch_h\in \Sig_h} \|\sig-\ch_h\|_{L^2(\Ome)}\\
&\lesssim (1+kh)^{\frac12} hk + (k+k^{-1}) h \\
&\lesssim (1+kh)^{\frac12} hk,
\end{align*}
which gives \eqref{e4.5c}. The proof is complete.
\end{proof}

\subsubsection{\bf Global error estimates for the LDG method \#1} \label{sec-4.1.2}

In the preceding subsection we have derived the error bounds
for $(u-\tilde{u}_h,\sig-\tilde{\sig}_h)$. By the decomposition
$u-u_h=(u-\tilde{u}_h) + (\tilde{u}_h-u_h)$ and
$\sig-\sig_h= (\sig-\tilde{\sig}_h)+ (\tilde{\sig}_h-\sig_h)$
and the triangle inequality, it suffices to get
error bounds for $(\tilde{u}_h-u_h,\tilde{\sig}_h-\sig_h)$.
We shall accomplish this task by exploiting the linearity
of the Helmholtz equation and using the stability estimate
for the LDG method \#1 obtained in Section \ref{sec-3.1}.

First, on noting that $(u,\sig)$ satisfies
\begin{align}\label{e4.6}
A_h(u,\sig;v_h,\ta_h) =F(v_h,\ta_h) \qquad \forall (v_h,\ta_h)\in V_h\times\Sig_h.
\end{align}
Subtracting \eqref{e3.1} from \eqref{e4.6} yields the following
error equation (or Galerkin orthogonality):
\begin{align}\label{e4.7}
A_h(u-u_h,\sig-\sig_h;v_h,\ta_h) =0 \qquad\forall (v_h,\ta_h)\in V_h\times\Sig_h.
\end{align}

Next, to proceed we introduce the notation
\begin{alignat*}{3}
&u-u_h=e_h +q_h, &&\qquad e_h:=u-\tilde{u}_h, &&\quad q_h:=\tilde{u}_h-u_h,\\
&\sig-\sig_h = \ps_h +\ph_h, &&\qquad \ps_h:=\sig-\tilde{\sig}_h,
&&\quad \ph_h:=\tilde{\sig}_h-\sig_h.
\end{alignat*}
Then by \eqref{e4.7} and the definitions of the sesquilinear form $a_h$
and the elliptic projection we have
\begin{align}\label{e4.8}
A_h(q_h,\ph_h;v_h,\ta_h) &= -A_h(e_h,\ps_h;v_h,\ta_h) \\
&=-a_h(e_h,\ps_h;v_h,\ta_h) + k^2 (e_h, v_h) \nonumber \\
&=k^2 (e_h, v_h), \qquad\forall (v_h,\ta_h)\in V_h\times\Sig_h. \nonumber
\end{align}
The above equation implies that $(q_h,\ph_h)\in V_h\times\Sig_h$
is the LDG solution to the Helmholtz problem with source terms
$f=k^2 e_h$ and $g=0$. Then an application of the stability estimates
of Theorem \ref{ldg1_sta} immediately yields the following lemma.

\begin{proposition}\label{prop4.2}
There hold the following estimates for $(q_h,\ph_h)$:
\begin{align}\label{e4.9}
\|q_h\|_{DG}  &\lesssim  \gamma_1 (1+kh) k^2 h^2, \\
\|\ph_h\|_{L^2(\Ome)} &\lesssim
\gamma_1 (1+kh +\delta_0 kh^2)(1+kh) k h.  \label{e4.10}
\end{align}
\end{proposition}

Combining Proposition \ref{prop4.1} and \ref{prop4.2}, using the triangle
inequality and the standard duality argument give the following
error estimates for $(u_h,\sig_h)$.

\begin{theorem}\label{thm4.1}
Let $u\in H^2(\Ome)$ be the solution to problem \eqref{helm-1}--\eqref{helm-2}
and $\sig:=\nab u$, and  $(u_h,\sig_h)$ be the solution to problem
\eqref{e3.1}. Then there hold the following error estimates for $(u_h,\sig_h)$:
\begin{align}\label{e4.11a}
&\|u-u_h\|_{1,h} +k^{\frac12} \|u-u_h\|_{L^2(\Gamma)}
\lesssim \bigl((1+kh)^{\frac12} + \gamma_1 (1+kh)kh \bigr) kh, \\
&\|u-u_h\|_{L^2(\Ome)} \lesssim (1+\gamma_1)(1+kh) k h^2, \label{e4.11b} \\
&\|\sig-\sig_h\|_{L^2(\Ome)} \lesssim \bigl((1+kh)^{\frac12}
+\gamma_1 (1+kh)(1+kh+\delta_0 kh^2)\bigr) k h. \label{e4.11c}
\end{align}
\end{theorem}

\subsection{Error estimates for the LDG method \#2} \label{sec-4.2}
The error analysis for the LDG method \#2 essentially follows the same
lines as that for the LDG method \#1 given in the previous subsection.
However, there are three main differences which we now explain.
First, the sesquilinear form $a_h$ needs to be
replaced by another sesquilinear form $b_h$ in the definition of the
elliptic projection \eqref{e4.1}, where $b_h$ is defined by
\begin{align}\label{e4.13}
b_h(w_h, \ch_h; v_h,\ta_h) :&=B_h(w_h,\ch_h;v_h,\ta_h) +k^2 (w_h,v_h)_\Ome \\
&=(\ch_h,\nabla_h v_h)_\Ome +\i k\Langle w_h,v_h\Rangle_\Ga \nonumber \\
&\qquad
-\sum_{e\in \cE_h^I}\Langle \{\ch_h\}-\i \beta [[w_h]],[[v_h]] \Rangle_e
\nonumber \\
&\qquad
-\sum_{e\in \cE_h^I} \Bigl( \i\delta\Langle [[\ch_h]],[[\ta_h]]\Rangle_e
-\Langle [[w_h]],\{\ta_h\} \Rangle_e \Bigr) \nonumber \\
&\qquad
+(\ch_h,\ta_h)_\Ome -(\nabla_h w_h,\ta_h)_\Ome. \nonumber
\end{align}

Second, due to strong coupling between $\tilde{u}_h$ and $\tilde{\sig}_h$,
the error estimates for the new elliptic projection $(\tilde{u}_h,\tilde{\sig}_h)$
must be derived differently. To the end, we need the following lemma,
which replaces Lemma \ref{lem4.1}.

\begin{lemma}\label{lem4.2}
Let $\beta=\beta_0 h^{-1}$ and $\delta=\delta_0 h$ for some positive
constants $\beta_0$ and $\delta_0$.

{\rm (i)} There exists an $h$- and $k$-independent constant $c_3>0$
such that the sesquilinear form $b_h$ satisfies the following generalized
{\em inf-sup} condition: for any fixed $(w_h, \ch_h)\in V_h\times \Sig_h$
\begin{align}\label{e4.13a}
&\sup_{(v_h,\ta_h)\in V_h\times \Sig_h}
\frac{\re b_h(w_h, \ch_h; v_h,\ta_h)}{|||(v_h,\ta_h)|||_{DG}} \\
&\hskip 0.8in
+ \sup_{(v_h,\ta_h)\in V_h\times \Sig_h}
\frac{\im b_h(w_h, \ch_h; v_h,\ta_h)}{|||(v_h,\ta_h)|||_{DG}}
\geq c_3 |||(w_h,\ch_h)|||_{DG}. \nonumber
\end{align}

{\rm (ii)} There exists an $h$- and $k$-independent constant $C>0$ such that
for any $(w, \ch),(v,\ta)\in H^2(\cT_h)\times H^1(\cT_h)^d$,  there holds
\begin{align}\label{e4.13b}
|b_h(w, \ch; v,\ta)|\leq C |||(w,\ch)|||_{1,h} |||(v,\ta)|||_{1,h},
\end{align}
where
\begin{align}\label{e4.13c}
|||(w,\ch)|||_{DG} &:= \Bigl(\|w\|_{1,h}^2+\|\ch\|_{L^2(\Ome)}^2 \Bigr)^{\frac12},\\
|||(w,\ch)|||_{1,h} &:= \Bigl(|||(w,\ch)|||_{DG}^2+ \sum_{e\in \cE_h^I}
\beta^{-1}\|\{\ch\}\|_{L^2(e)}^2 \Bigr)^{\frac12}. \label{e4.13d}
\end{align}
\end{lemma}

The proof of (i) is based on evaluating the first quotient on the left-hand
side of \eqref{e4.13a} at $(v_h,\ta_h)= \bigl((1+C_1)w_h,C_1\ch_h-\nab_h w_h\bigr)$
and evaluating the second quotient at $(v_h,\ta_h)= (C_2w_h,C_2\ch_h)$
for some sufficiently large positive constants $C_1$ and $C_2$.
The proof of (ii) is a straightforward application of Schwarz and
trace inequalities. We skip the rest of the derivation to save space.

The above generalized {\em inf-sup} condition, the boundedness
of the sesquilinear form $b_h$, and the duality argument
(cf. \cite{Brenner_Scott08}) readily infer the
following error estimates for the new elliptic projection
$(\tilde{u}_h,\tilde{\sig}_h)$. We omit the proof since it is standard.

\begin{proposition}\label{prop4.3}
Under the assumptions of Proposition \ref{prop4.1}, there hold the following
estimates:
\begin{align}\label{e4.13e}
\|u-\tilde{u}_h\|_{1,h} +\|\sig-\tilde{\sig}_h\|_{L^2(\Ome)} &\lesssim kh, \\
\|u-\tilde{u}_h\|_{L^2(\Ome)} &\lesssim  k^2h^2. \label{e4.13f}
\end{align}
\end{proposition}

The third difference is that the new error function $(q_h,\ph_h)$ now satisfies
\begin{align}\label{e4.14}
B_h(q_h,\ph_h;v_h,\ta_h) =k^2 (e_h, v_h)
\qquad\forall (v_h,\ta_h)\in V_h\times\Sig_h.
\end{align}
As a result, by Theorem \ref{ldg2_sta} and \eqref{sec3.2:lem1} we get
\begin{align}\label{e4.15}
|q_h|_{1,h}+\|(q_h,\ph_h)\|_{DG} \lesssim \gamma_2 (1+kh) k^2 h^2,
\end{align}
which replaces estimates \eqref{e4.9} and \eqref{e4.10}.

After having established Proposition \ref{prop4.3} and
\eqref{e4.15}, once again, by the triangle inequality we arrive at
the following error estimates for the solution $(u_h,\sig_h)$
to the LDG method \#2.

\begin{theorem}\label{thm4.2}
Let $u\in H^2(\Ome)$ be the solution to problem \eqref{helm-1}--\eqref{helm-2}
and $\sig:=\nab u$, and  $(u_h,\sig_h)$ be the solution to problem
\eqref{e3.2.1}. Then there hold the following error estimates for $(u_h,\sig_h)$:
\begin{align}\label{e4.16}
&\|u-u_h\|_{1,h} +k^{\frac12} \|u-u_h\|_{L^2(\Gamma)}  \\
&\hskip 0.7in
+\|\sig-\sig_h\|_{L^2(\Ome)} \lesssim \bigl(1+\gamma_2 (1+kh)kh)\bigr) kh, \nonumber\\
&\|u-u_h\|_{L^2(\Ome)} \lesssim (1+\gamma_2 (1+kh)) k^2 h^2. \label{e4.17}
\end{align}
\end{theorem}

\begin{remark}
\eqref{e4.15} shows that $\ph_h:=\tilde{\sig}_h-\sig_h$ has an optimal order (in $h$)
error bound for the LDG method \#2, while \eqref{e4.10} shows that $\ph_h$ only
has a sub-optimal order error bound for the LDG method \#1. We believe that this
is the main reason why in practice the LDG method \#2 gives a better approximation to
the flux variable $\sig$ than the LDG method \#1 does although both methods have
the same asymptotic rate of convergence in $h$.
\end{remark}

\section{Numerical experiments} \label{sec-5}

In this section we shall provide some numerical results of the two proposed
LDG methods. Our tests are done for the following $2$-d Helmholtz problem:
\begin{alignat}{2}
-\Del u - k^2 u &=f:=\frac{\sin(kr)}{r}  &&\qquad\mbox{in  } \Omega,\label{e5.1}\\
\frac{\pa u}{\pa n_\Omega} +\i k u &=g &&\qquad\mbox{on } \Ga_R:=\pa\Omega.\label{e5.2}
\end{alignat}
Here $\Omega$ is the unit square $[-0.5,\,0.5]\times [-0.5,\,0.5]$,
and $g$ is chosen so that the exact solution is given by
\begin{equation}\label{e5.3}
u=\frac{\cos(kr)}{k}-\frac{\cos k+\i\sin k}{k\big(J_0(k)+\i J_1(k)\big)}J_0(kr)
\end{equation}
in polar coordinates, where $J_\nu(z)$ are Bessel functions of the first kind.

Assume $\cT_{1/m}$ be the regular triangulation that consists of
$2m^2$ right-angled equicrural triangles of size $h=1/m$, for any positive integer
$m$. See Figure \ref{fig1} for the sample triangulation $\cT_{1/4}$ and $\cT_{1/10}$.

\begin{figure}[ht]
\centerline{
\includegraphics[width=2.2in]{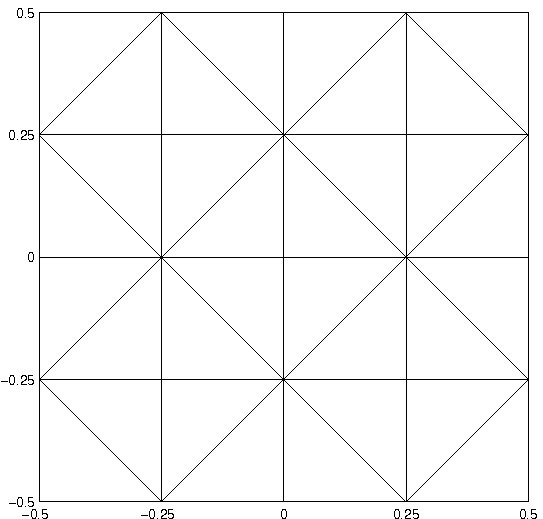}
\includegraphics[width=2.2in]{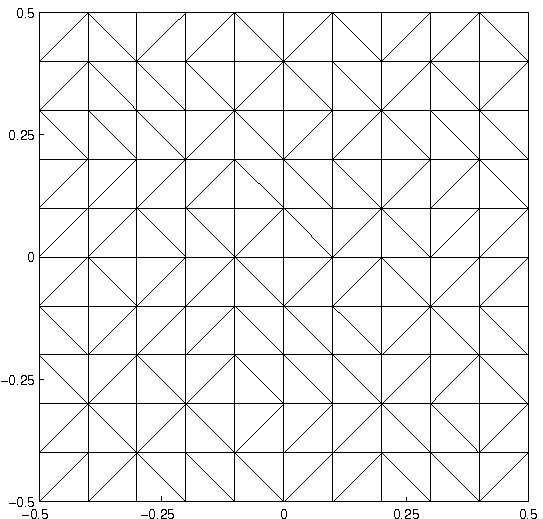}}
\caption{The computational domain and sample meshes.
Left: $\cT_{1/4}$ that consists of right-angled equicrural triangles
of size $h=\frac14$; Right: $\cT_{1/10}$ with $h=\frac{1}{10}$.}\label{fig1}
\end{figure}

\subsection{Sensitivity with respect to the parameters $\delta$
and $\beta$}\label{sec5.1}

In this subsection, we examine the sensitivity of the error of the
LDG solutions in $H^1$-seminorm with respect to the parameters
$\delta$ and $\beta$.

The LDG method \#1 is considered first.
We start by fixing $\delta=0.1 h_e$ and testing the sensitivity in the
parameter $\beta$. With two wave numbers $k=5$ and $50$,
we compute the solutions of the LDG method \#1
with different values of $\beta$: $0.001 h_e^{-1}$,
$0.01 h_e^{-1}$, $h_e^{-1}$ and $1$. The relative errors,
defined by the errors in the $H^1$-seminorm divided by the
exact solution in the $H^1$-seminorm, are shown in the left graph of
Figure \ref{fig5.1.1}. We observe that the relative errors
have similar behaviors and decay as mesh size
$h$ becomes smaller. This shows that the errors are not sensitive to
the parameter $\beta$. Next, we fix $\beta=0.001 h_e^{-1}$, and repeat
the test with different $\delta$. The right graph of Figure \ref{fig5.1.1}
shows the relative errors with parameters $\delta = 0.001 h_e$, $0.1 h_e$, $10 h_e$
and $0.1$, and wave numbers $k=5$ and $50$. We observe that the errors have
similar behaviors for small values $\delta=0.001 h_e$ and $0.1 h_e$. Larger
$\delta$ results in larger error.

The sensitivity tests of the LDG method \#2 are shown in Figure \ref{fig5.1.2},
similar behaviors are also observed.

\begin{figure}[ht]
\centerline{
\includegraphics[width=2.5in]{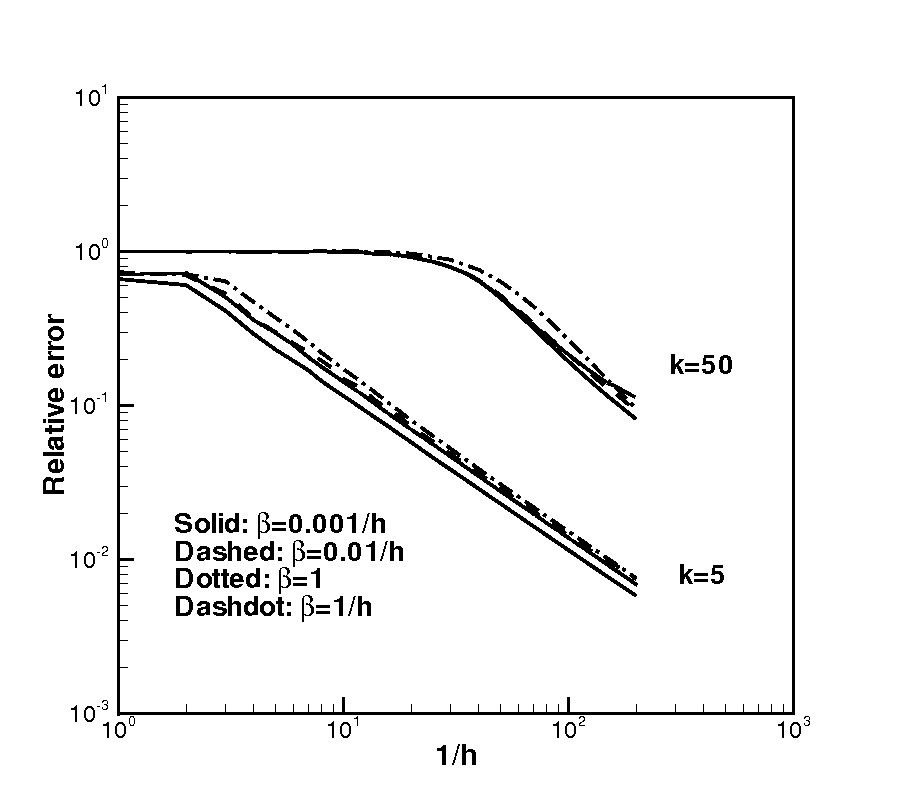}
\includegraphics[width=2.5in]{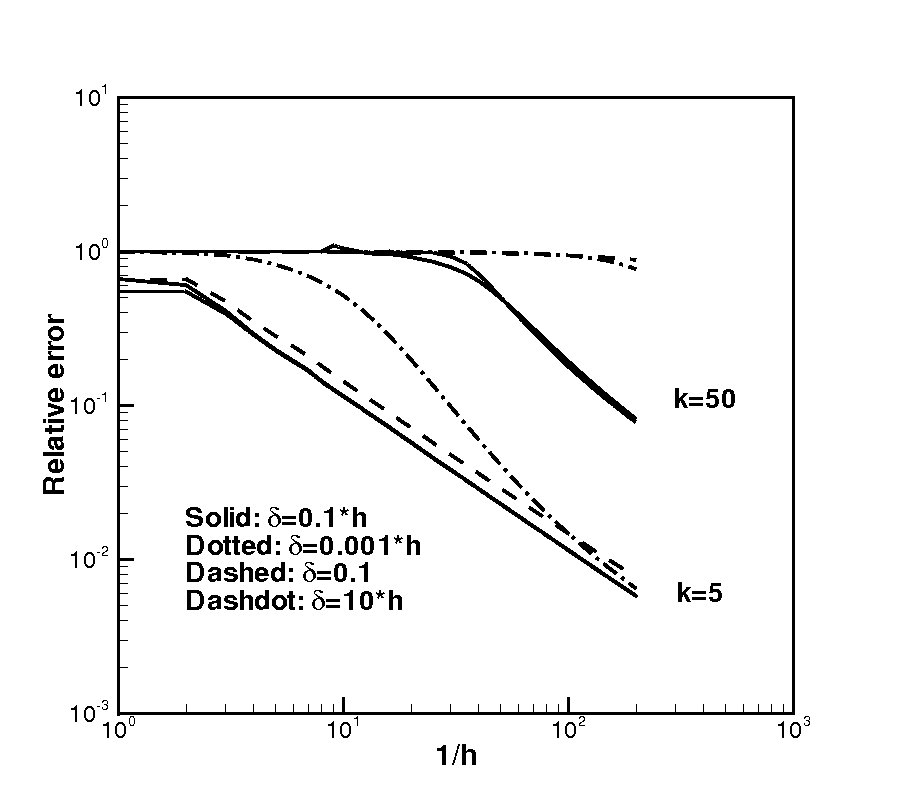}}
\caption{Relative error in the $H^1$-seminorm of the LDG method \#1 with different
parameters for two wave numbers $k=5$ and $50$.
Left: $\delta=0.1 h_e$ is fixed, $\beta = 0.001 h_e^{-1}$,
$0.01 h_e^{-1}$,  $h_e^{-1}$ and $1$;
Right: $\beta=0.001 h_e^{-1}$ is fixed, $\delta = 0.001 h_e$,
$0.1 h_e$, $10 h_e$ and $0.1$. }\label{fig5.1.1}
\end{figure}

\begin{figure}[ht]
\centerline{
\includegraphics[width=2.5in]{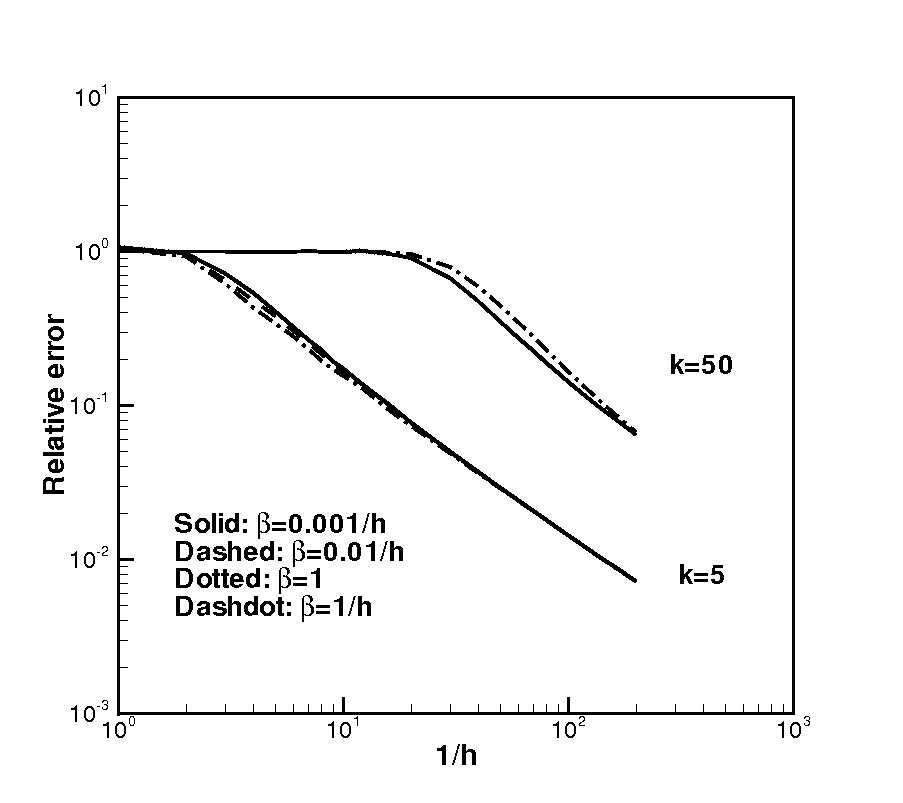}
\includegraphics[width=2.5in]{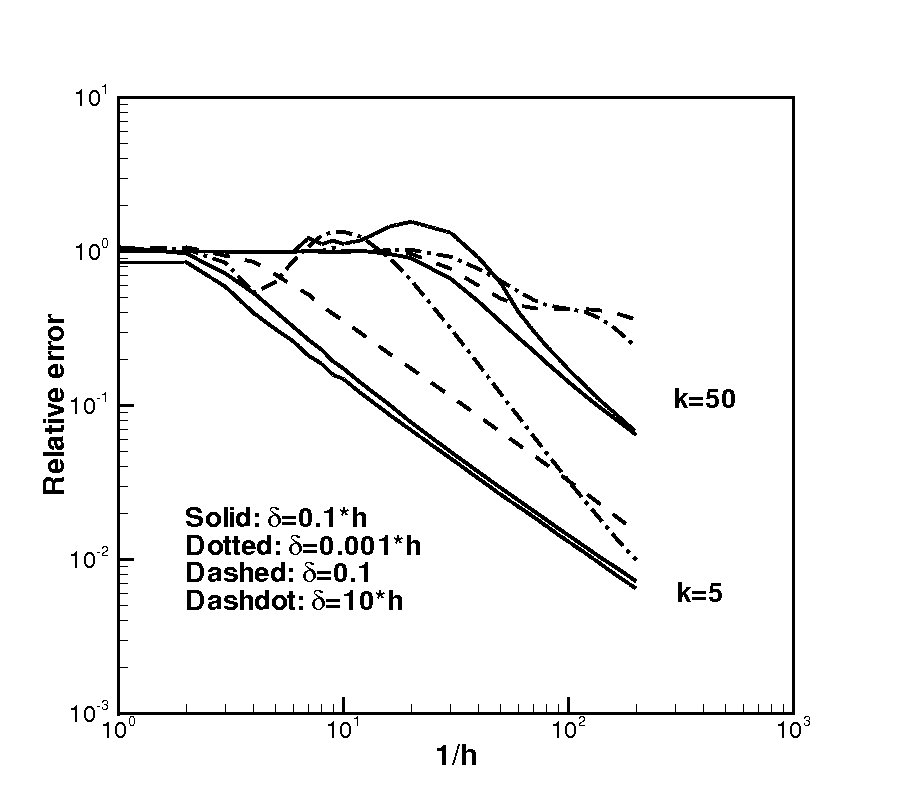}}
\caption{Relative error in the $H^1$-seminorm of the LDG method \#2 with different
parameters for two wave numbers $k=5$ and $50$.
Left: $\delta=0.1 h_e$ is fixed, $\beta = 0.001 h_e^{-1}$,
$0.01 h_e^{-1}$,  $h_e^{-1}$ and $1$;
Right: $\beta=0.001 h_e^{-1}$ is fixed, $\delta = 0.001 h_e$,
$0.1 h_e$, $10 h_e$ and $0.1$. }\label{fig5.1.2}
\end{figure}

\subsection{Errors of the LDG solutions}\label{sec5.2}

In this subsection, we fix the parameters and investigate the
changes of the numerical errors as functions of the mesh size.

We start from the LDG method \#1. As suggested by the sensitivity
tests in the previous subsection, we pick
\begin{equation}\label{ldg1:par}
\delta = 0.1 h_e, \qquad  \beta = 0.001 h_e^{-1}.
\end{equation}
The relative error of the LDG method, and the finite element interpolation
are shown in the left graph of Figure \ref{fig5.2.1}, with four different
wave numbers $k=5$, $10$, $50$ and $100$.  The relative error of the
LDG solution stays around $100\%$ before a critical mesh size is reached,
then decays at a rate greater than $-1$ in the log-log scale but converges
as fast as the finite element interpolation (with slope $-1$) for small $h$.
The critical mesh size decreases as $k$ increases.

The right graph of Figure \ref{fig5.2.1} contains
the relative error when we fix $k h=1$ and $hk=0.5$. It indicates that
unlike the error of the finite element interpolation
the error of the LDG is not controlled by the magnitude of $k h$,
which suggests that there is a pollution contribution in the total error.
The left graph of Figure \ref{fig5.2.2} contains the relative error
of the LDG method with the mesh size satisfying $k^3h^2=1$ for different
values of $h$. The error does not increase with respect to $k$.

The LDG method \#2 has also been tested using the same parameters in
\eqref{ldg1:par}. Similar behaviors have been observed as shown in Figure
\ref{fig5.2.2} and \ref{fig5.2.3}.

\begin{figure}[ht]
\centerline{
\includegraphics[width=2.5in]{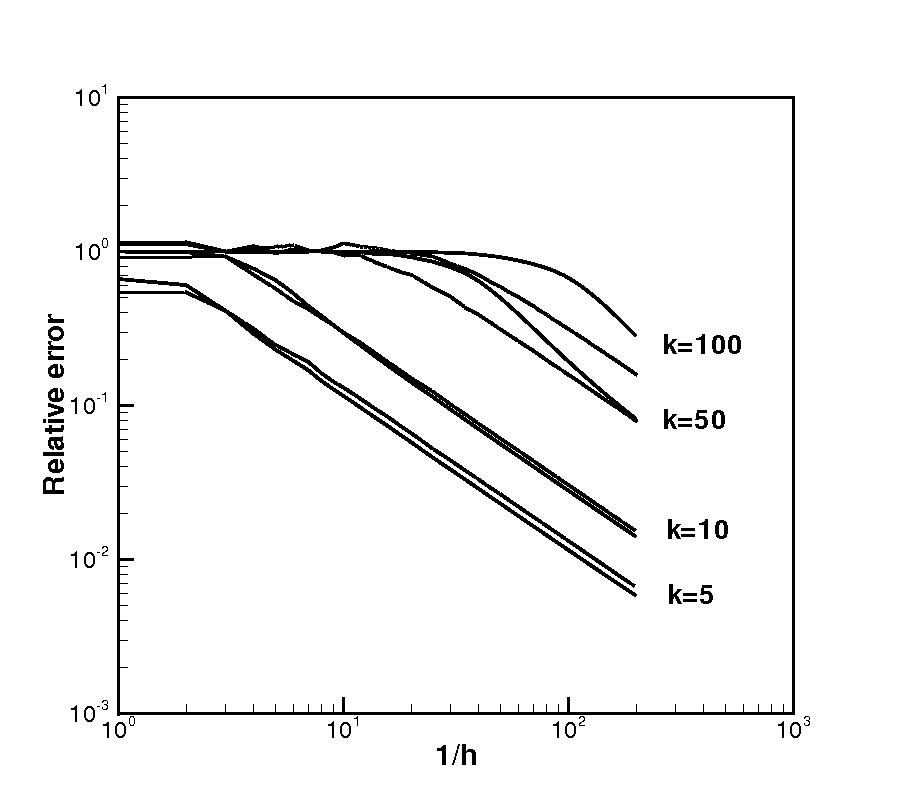}
\includegraphics[width=2.5in]{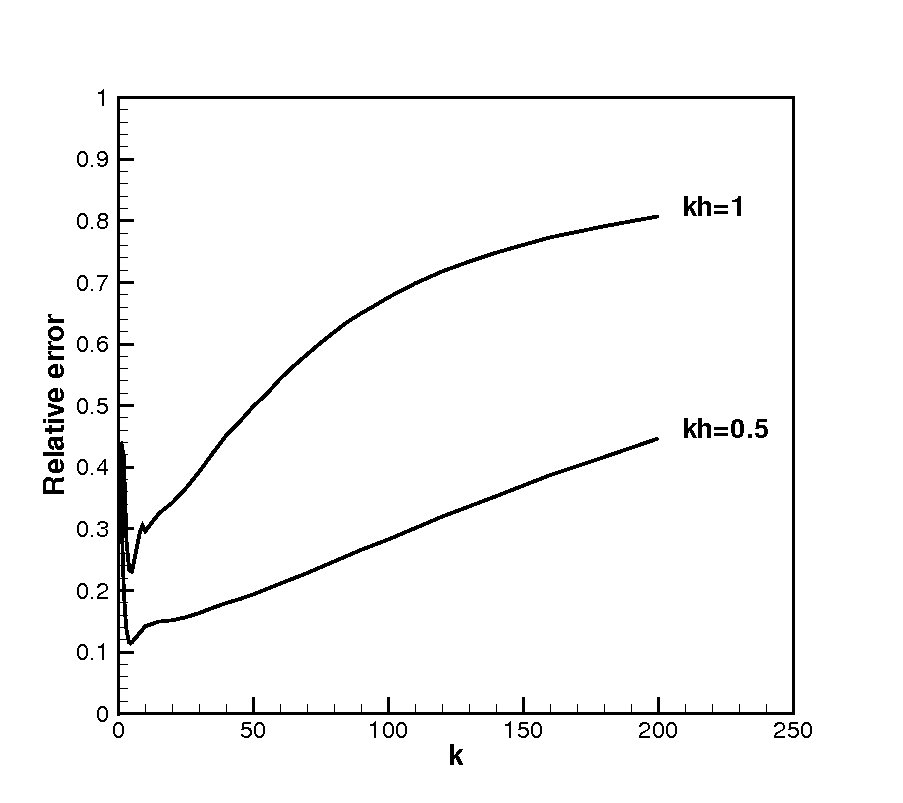}}
\caption{Left: relative error of the LDG method \#1 (solid line)
and the finite element interpolation (dotted line) in the $H^1$-seminorm for
$k=5$, $10$, $50$ and $100$; Right: relative error of the  LDG method \#1
in the $H^1$-seminorm for $k=1,\cdots,200$, $kh=1$ and $kh=0.5$.}\label{fig5.2.1}
\end{figure}

\begin{figure}[ht]
\centerline{
\includegraphics[width=2.5in]{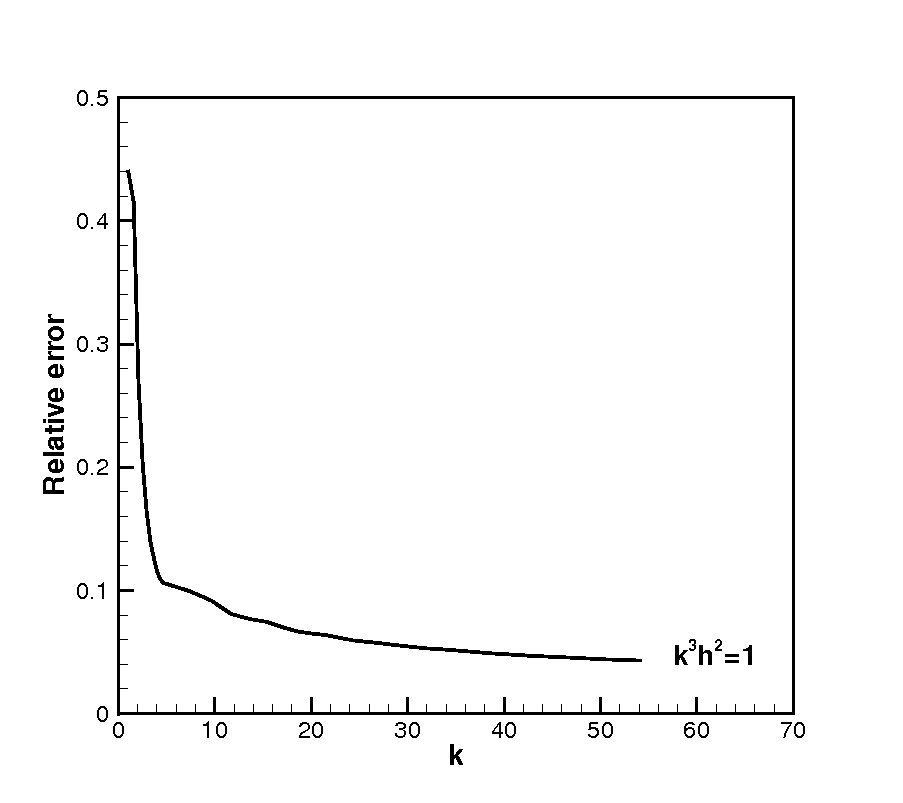}
\includegraphics[width=2.5in]{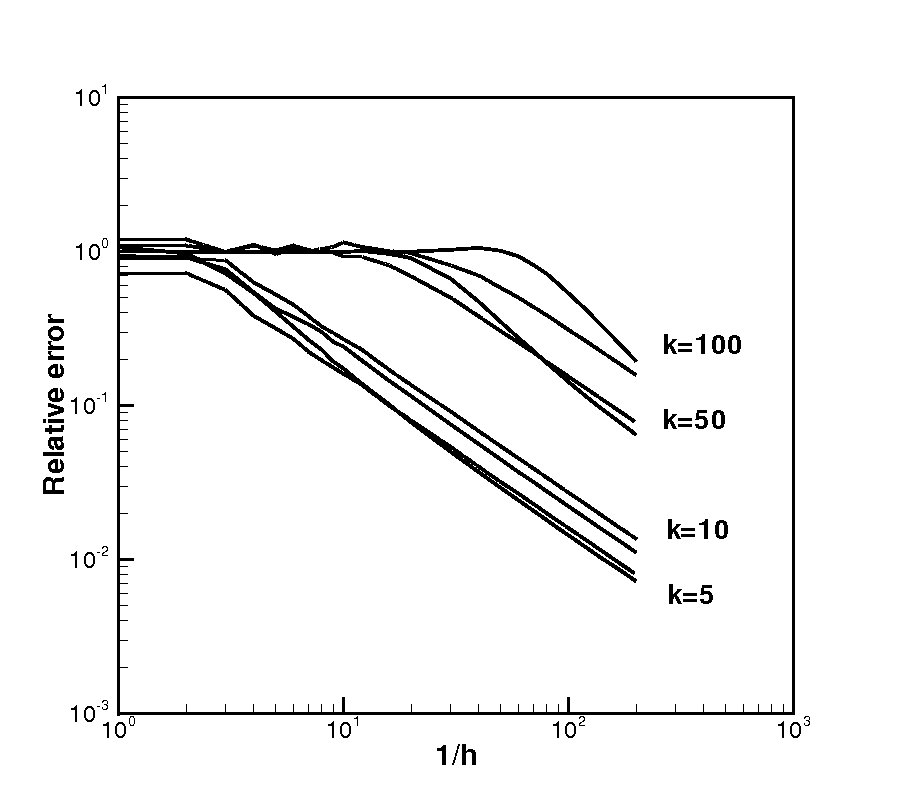}}
\caption{Left: relative error of the  LDG method \#1
in the $H^1$-seminorm with $k^3h^2=1$;
Right: relative error of the LDG method \#2 (solid line)
and the finite element interpolation (dotted line) in the $H^1$-seminorm for
$k=5$, $10$, $50$ and $100$.}\label{fig5.2.2}
\end{figure}

\begin{figure}[ht]
\centerline{
\includegraphics[width=2.5in]{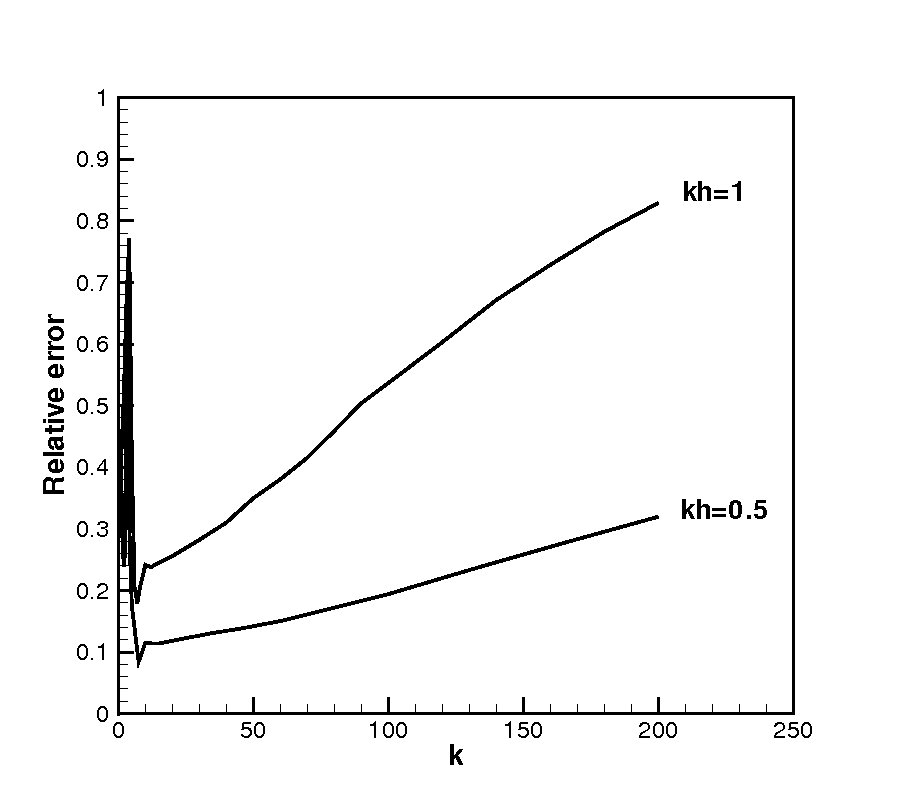}
\includegraphics[width=2.5in]{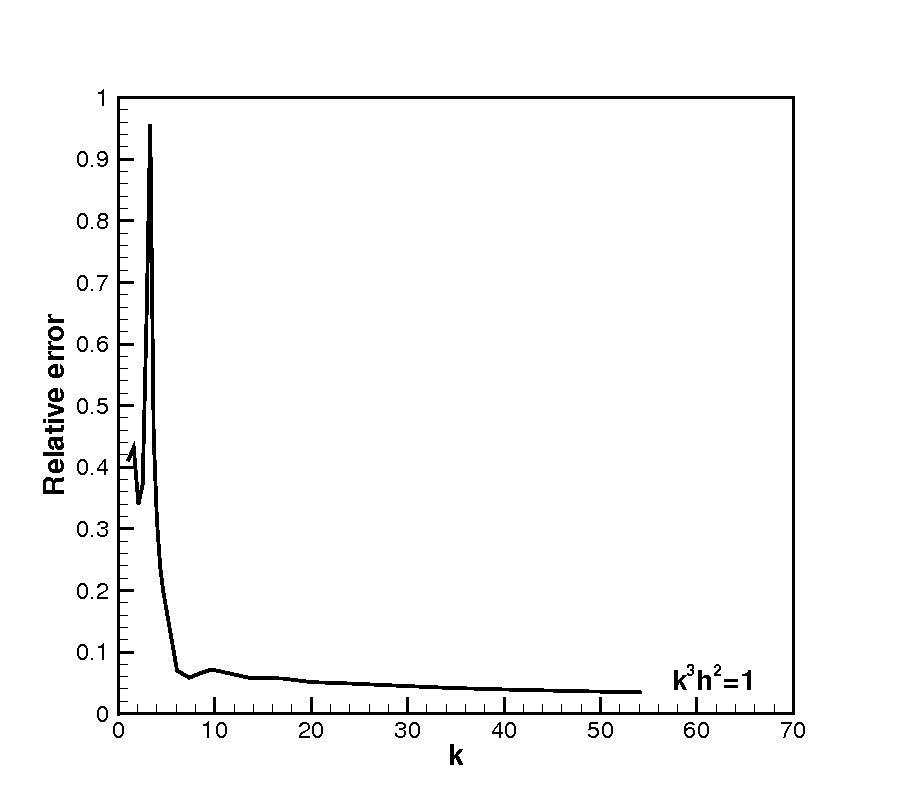}}
\caption{Relative error of the LDG method \#2 in the $H^1$-seminorm.
Left: $kh=1$ and $kh=0.5$ Right: $k^3h^2=1$.}\label{fig5.2.3}
\end{figure}

At the end, we look closely at the situation with a large relative error
when $kh>1$. The LDG method \#1 solution with parameters $\delta=0.1 h_e$,
$\beta=0.001 h_e^{-1}$, $k=100$ and $h=1/45$ has a large relative error
of size $0.9392$. The surface plots of the finite element interpolation and the LDG
solution are given in Figure \ref{fig5.2.4}. It shows that the LDG solution
has the correct shape/phase although its amplitude is smaller.

\begin{figure}[ht]
\centerline{
\includegraphics[width=2.5in]{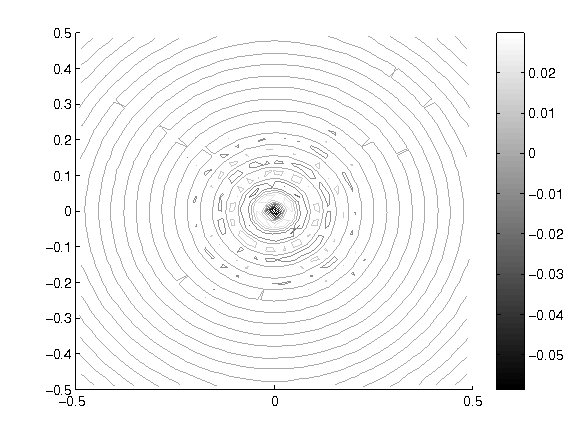}
\includegraphics[width=2.5in]{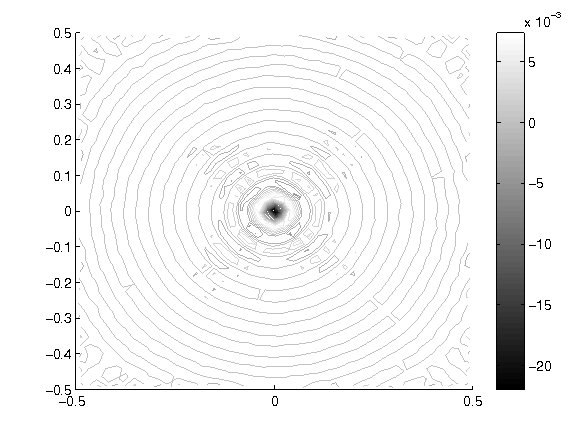}}
\caption{Left: surface plots of the finite element interpolation (left)
and the LDG method \#1 solution (right) with parameters
$\delta=0.1 h_e$, $\beta=0.001 h_e^{-1}$, $k=100$ and $h=1/45$.}\label{fig5.2.4}
\end{figure}

\subsection{Comparison between the two LDG methods}\label{sec5.3}

Two different LDG methods are proposed in this paper. The first
one is derived following the IPDG method proposed in \cite{fw08a},
and the second one has a more standard numerical flux formulation
and is supposed to have a better approximation for the vector/flux variable.
In this subsection, we provide a comparison between these two methods,
in terms of the error and computational cost.

We start by revisiting the test examples of Subsection \ref{sec5.1}.
Instead of computing the relative error of $u_h$ in the $H^1$-seminorm,
we compute the relative error of $\sig_h$ in the $L^2$-norm for the
LDG method \#2. The numerical results are presented in Figure \ref{fig5.3},
which show that although the solution is still not
sensitive to the parameter $\beta$, better approximation to $\sig$
is achieved for larger $\delta$. It confirms our prediction that the
LDG method \#2 gives a better approximation for the vector/flux variable.

\begin{figure}[ht]
\centerline{
\includegraphics[width=2.5in]{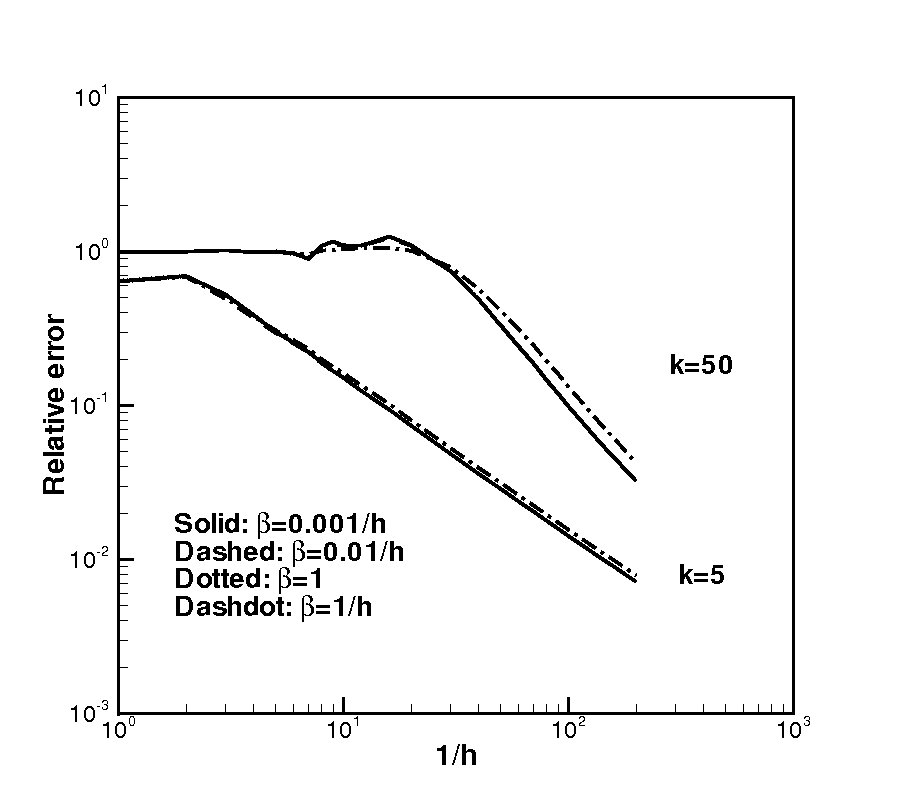}
\includegraphics[width=2.5in]{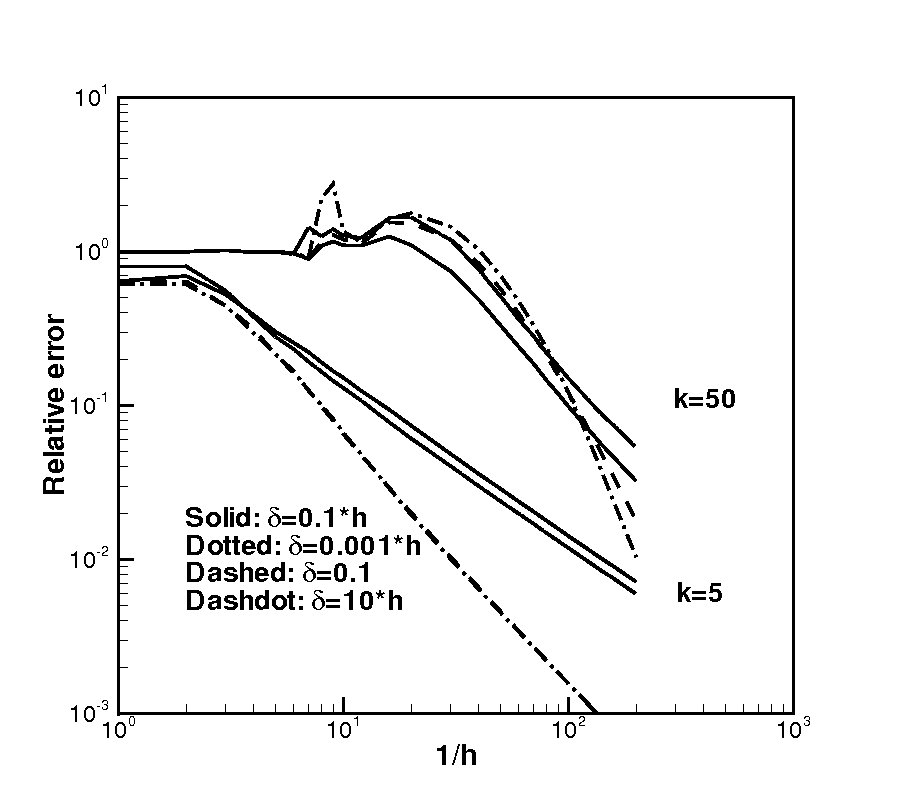}}
\caption{Relative error of $\sig_h$ in $L^2$-norm
of the LDG method \#2 with different
parameters for two wave numbers $k=5$ and $50$.
Left: $\delta=0.1 h_e$ is fixed, $\beta = 0.001 h_e^{-1}$,
$0.01 h_e^{-1}$,  $h_e^{-1}$ and $1$;
Right: $\beta=0.001 h_e^{-1}$ is fixed, $\delta = 0.001 h_e$,
$0.1 h_e$, $10 h_e$ and $0.1$. }\label{fig5.3}
\end{figure}

Table \ref{tab1} provides a
detailed comparison of these two methods for different mesh sizes
$h$, with the parameters $\delta =0.1 h_e$, $\beta=0.001 h_e^{-1}$ and $k=10$.
It shows that the computational cost of the LDG method \#2 is about twice larger than
that of the LDG method \#1. Also, as expected, the error of the vector/flux variable
of the LDG method \#2 is smaller than that of the LDG method \#1, and both methods
demonstrate a first order rate of convergence.

\begin{table}
\caption{Comparison of the two LDG methods with parameters
$\delta =0.1 h_e$ and $\beta=0.001 h_e^{-1}$.} \label{tab1}
\begin{tabular}{|c|c|c|c|c|c|c|}
\hline
& $1/h$ & $|u-u_h|_{H^1}$ & order & $\|\sig -\sig_h\|_{L^2}$
& order & CPU time (s) \\ \hline
& 5 & 4.1059E-01 &  & 5.4715E-01 &  & 0.0641 \\ \cline{2-7}
& 10 & 1.6915E-01 & 1.2794 & 2.4712E-01 & 1.1467 & 0.2381 \\ \cline{2-7}
LDG \#1 & 20 & 7.6089E-02 & 1.1525 & 1.1804E-01 & 1.0659 & 0.9671 \\
\cline{2-7}
& 40 & 3.6648E-02 & 1.0539 & 5.7114E-02 & 1.0474 & 3.9380 \\ \cline{2-7}
& 80 & 1.8151E-02 & 1.0137 & 2.8319E-02 & 1.0121 & 15.8194 \\ \cline{2-7}
& 160 & 9.0379E-03 & 1.0060 & 1.4004E-02 & 1.0159 & 69.1861 \\ \hline\hline
& 5 & 2.4711E-01 &  & 2.2184E-01 &  & 0.1057 \\ \cline{2-7}
& 10 & 1.4040E-01 & 0.8156 & 7.6775E-02 & 1.5308 & 0.4368 \\ \cline{2-7}
LDG \#2 & 20 & 6.6992E-02 & 1.0675 & 3.3630E-02 & 1.1909 & 1.8356 \\
\cline{2-7}
& 40 & 3.2693E-02 & 1.0350 & 1.5710E-02 & 1.0981 & 7.9357 \\ \cline{2-7}
& 80 & 1.6165E-02 & 1.0161 & 7.7418E-03 & 1.0209 & 34.5688 \\ \cline{2-7}
& 160 & 8.0949E-03 & 0.9978 & 3.9127E-03 & 0.9845 & 157.6018 \\ \hline
\end{tabular}
\end{table}

\begin{figure}[hbt]
\centerline{
\includegraphics[width=2.5in,height=1.2in]{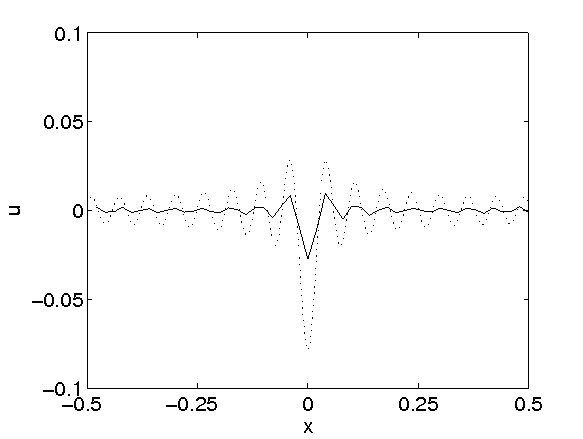}
\hspace{0.05in}
\includegraphics[width=2.5in,height=1.2in]{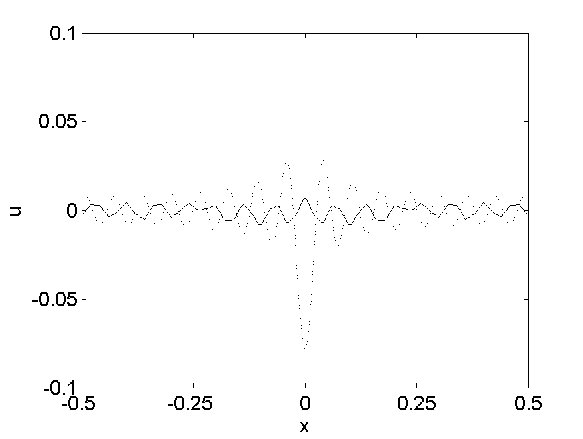}
} \centerline{
\includegraphics[width=2.5in,height=1.2in]{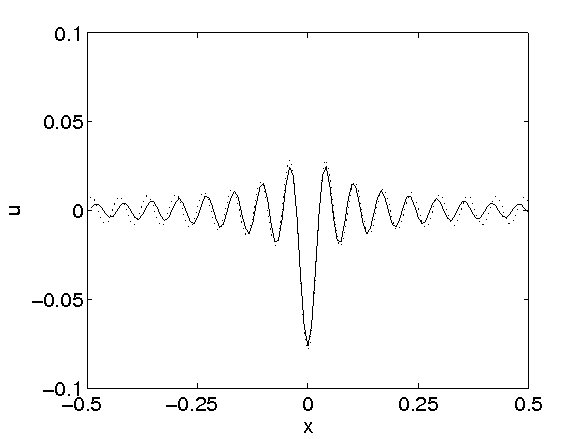}
\hspace{0.05in}
\includegraphics[width=2.5in,height=1.2in]{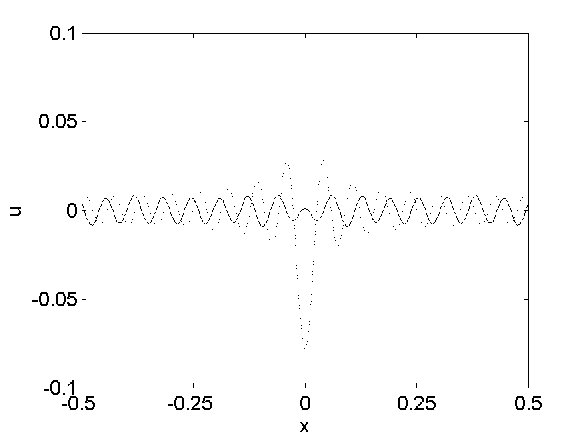}
} \centerline{
\includegraphics[width=2.5in,height=1.2in]{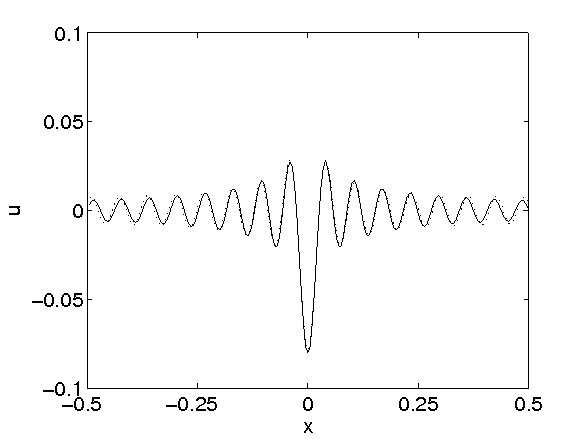}
\hspace{0.05in}
\includegraphics[width=2.5in,height=1.2in]{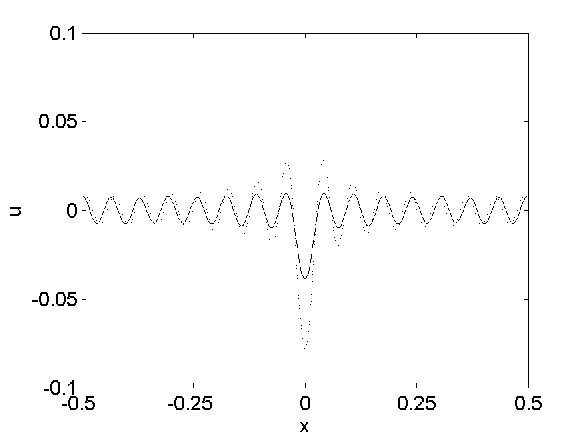}
} \caption{The traces of the LDG \#1 solution (left) and the finite element
solution (right) in the $xz$-plane, for $k=100$ and $h=1/50$ (top), $1/120$ (middle)
and $1/200$ (bottom), respectively. The dotted lines are the traces of the exact
solution.} \label{fig5.4}
\end{figure}

\subsection{Comparison between LDG and finite element solutions}

We have shown the performance and comparison of the two LDG methods
in previous subsections. In this subsection, we provide a brief
comparison between the LDG solution and the $P_1$ conforming finite element solution.

We consider the Helmholtz problem \eqref{e5.1}-\eqref{e5.2} with wave number
$k=100$. With mesh size $h=1/50$, $1/120$ and $1/200$, we plot the traces of the
LDG method \#1 solution with parameters \eqref{ldg1:par} in $xz$-plane in
the left column of Figure \ref{fig5.4}. The exact solution is also
provided as a reference. The traces of the finite element
solution are shown in the right column of Figure \ref{fig5.4}.
It is clear that the LDG method \#1 has a better approximation
to the exact solution. Larger phase error in
the finite element solution is observed in all three cases.
Also, the LDG solution has a better approximation for the amplitude
of the exact solution.


\end{document}